\documentclass[a4paper,10pt,leqno]{scrartcl}
\usepackage[utf8]{inputenc}
\usepackage[USenglish]{babel}

\usepackage{amsmath,amssymb,amsfonts,amsthm,amsgen}
\usepackage[matrix,arrow,curve,frame]{xy}
\usepackage{graphicx}
\usepackage{color}

\usepackage{enumitem} 

\usepackage{booktabs}

\addtokomafont{captionlabel}{\bfseries}
\setcapindent{0.5em}
\deffootnote[1em]{0.5em}{1em}{\textsuperscript{\thefootnotemark}}

\usepackage{float}

\usepackage{todonotes}

\usepackage{bbm}

\numberwithin{equation}{section} \numberwithin{figure}{section}
\numberwithin{table}{section}

        \newcommand{\Z}{\mathbb{Z}}

        \newcommand{\RP}{\mathbbm{R}\mathrm{P}}
        \newcommand{\R}{\mathbbm{R}}

        \newcommand{\lt}{\left(}
        \newcommand{\rt}{\right)}

        \newcommand{\incl}{\operatorname{incl}}
        \newcommand{\id}{\operatorname{id}}

        \newcommand{\MC}{\operatorname{MC}}
        \newcommand{\MCC}{\operatorname{MCC}}
        \newcommand{\MF}{\operatorname{MF}}
        \newcommand{\Ker}{\operatorname{Ker}}
        \newcommand{\quot}{\operatorname{quot}}
        \newcommand{\pinch}{\operatorname{pinch}}
        \newcommand{\coll}{\operatorname{b}}
        \newcommand{\ST}{\operatorname{ST}}
        \newcommand{\comp}{\, \scriptstyle \circ \displaystyle}
        \newcommand{\nonComm}{\amalg}
        \newcommand{\KoschKursiv}[1]{\ensuremath{\mathit{#1}}}
        \newcommand{\KoschScript}[1]{\ensuremath{\mathcal{#1}}}
        \newcommand{\hochdot}{\boldsymbol{\cdotp}}

        \newtheorem{thm}[equation]{Theorem}
        \newtheorem{cor}[equation]{Corollary}
        \newtheorem{lem}[equation]{Lemma}
        \newtheorem{prop}[equation]{Proposition}
        
  \theoremstyle{definition}

  \theoremstyle{remark}
        
        \newtheorem{rem}[equation]{Remark}
  \newtheoremstyle{example}{3pt}{3pt}{}{}{\bfseries}{:}{.5em}{}
  \theoremstyle{example}
        \newtheorem{exa}[equation]{Example}
        
        \newtheorem{problem}[equation]{Problem}

\newenvironment{con_and_not}
  {\vspace{1ex}{\itshape {\bfseries Conventions and notations.}}}%
  {}

\title{Nielsen coincidence numbers, Hopf invariants and spherical space forms}
\author{Ulrich Koschorke}

\begin{document}
\addtolength{\parindent}{2pt}

\date{}
\maketitle

\newcommand{\dedication}[1]{\begin{center}\itshape#1\normalfont\end{center}}
\dedication{Dedicated to Karl--Otto St{\"o}hr on occasion of his
70\textsuperscript{th} birthday.}

\newcommand{\classno}[1]{\footnotetext{\normalfont\itshape Mathematics Subject Classification\/ #1}}
\classno{54H25, 55M20 (primary), 55Q25, 55Q40, 57R90 (secondary)}

\newcommand{\extraline}[1]{\footnotetext{#1}}
\extraline{This work was supported by DFG (Deutsche
Forschungsgemeinschaft)}

\begin{abstract}
Given two maps between smooth manifolds, the obstruction to removing
their coincidences (via homotopies) is measured by minimum numbers.
In order to determine them we introduce and study an infinite
hierarchy of Nielsen numbers $\ N_i, \ i\, = \, 0, 1, \ldots, \infty
\, $. They approximate the minimum numbers from below with
decreasing accuracy, but they are (in principle) more easily
computable as $\ i\ $ grows. If the domain and the target manifold
have the same dimension (e.g. in the fixed point setting) all these
Nielsen numbers agree with the classical definition. However, in
general they can be quite distinct.

While our approach is very geometric the computations use the
techniques of homotopy theory and, in particular, all versions of
Hopf invariants (\`a la Ganea, Hilton, James..). As an illustration
we determine all Nielsen numbers and minimum numbers for pairs of
maps from spheres to spherical space forms. Maps into even
dimensional real projective spaces turn out to produce particularly
interesting coincidence phenomena.
\end{abstract}

\newcommand{\acknowledgementsname}{{\bfseries Acknowledgements}}
\newenvironment{acknowledgements}%
  {\vspace{1ex}\itshape\acknowledgementsname. \normalfont}%
  {}
\newenvironment{acknowledgement}{\renewcommand{\acknowledgementsname}{{\bfseries Acknowledgement}}%
  \begin{acknowledgements}}{\end{acknowledgements}}

\section{Introduction}

\indent

Consider (continuous) maps $\ f_1, f_2 \; \colon \; X
\longrightarrow Y \ $ between connected smooth manifolds without
boundary, $\ X \ $ being compact. We are interested in '{\itshape
essential}' aspects of their coincidence set
\begin{equation}\label{equ:1.1}
 C \; = \; C(f_1, f_2) \ := \ \lbrace x \; \in \; X \; \vert\; f_1(x) \; = \; f_2(x) \rbrace \, ,
\end{equation}
i.e. in those features which are preserved by homotopies $\ f_i \sim
f_i' ,\  i=1,2 \,$. Such essential phenomena can be measured to some
extend by the minimum numbers (of coincidence points and
pathcomponents, resp.)
\begin{equation}\label{equ:1.2}
 \MC(f_1, f_2) \; := \; \min\left\lbrace \# C(f_1' , f_2') \vert f_1' \sim f_1, f_2' \sim f_2 \right\rbrace
\end{equation}
and (even better)
\begin{equation}\label{equ:1.3}
 \MCC(f_1, f_2) \; := \; \min\left\lbrace \# \pi_0\left( C (f_1' , f_2') \right) \vert f_1' \sim f_1, f_2' \sim f_2 \right\rbrace \, .
\end{equation}
E.g. both numbers vanish precisely if the maps $\ f_1 , f_2 \ $ can
be deformed until they are coincidence free.

\begin{exa}\label{exa:1.4}
   {\bfseries topological fixed point theory.}
   Here
   $\ X \; = \; Y \, , \ f_2 \; = \; \id_X \, ,$
   and the principal object of study is the minimum number of fixed points
   \begin{equation*}
    \MF(f) \; := \; \min\left\lbrace \# C \left(f' , \id_X \right) \vert f' \sim f \right\rbrace \; = \; \MC(f, \id_X)
  \end{equation*}
  for maps
  $\ f \; \colon \; X \longrightarrow X \ $
  (cf [B], p.9; see also \cite{Br}).
  If it vanishes then so does the Lefschetz number, but the converse conclusion fails to hold in general.
  A powerful tool for a better understanding of minimum numbers was introduced by Jakob Nielsen in the 1920s when he described a lower bound
  $\ N(f) \ $
  of
  $\ \MC(f, \id_X) \, $.
  This 'Nielsen number' turned out to coincide with the minimum fixed point number precisely if
  $\ X \ $
  is not a surface with strictly negative Euler characteristic.
  (For an account of the spectacular history of this result see \cite{B}).
  \qed
\end{exa}

In general coincidence theory the geometry of generic coincidence
phenomena is much richer. E.g. when $\ X^m, Y^n \ $ are smooth
manifolds of dimensions $\ m > n \, $, then $\ C \ $ is generically
an ($m-n$)--manifold (and not just a finite set of isolated points).

In this paper we introduce an infinite hierarchy of (integer)
Nielsen coincidence numbers
\begin{equation}\label{equ:1.5}
  \left( \MC \, \geq \, \MCC \, \geq \, \right)\ N^{\#} \, \equiv \; N_0 \geq N_1 \geq N_2 \geq \ldots \geq N_r \geq \ldots \geq N_{\infty} \; \equiv \; \widetilde{N} \geq \, 0 \, .
\end{equation}
It interpolates between the sharpest ("nonstabilized") Nielsen
number $\ N_0 \; := \; N^{\#} \ $ introduced in \cite{K3} and the
("fully stabilized") Nielsen number $\ N_{\infty} \; := \;
\widetilde{N} \ $ (cf. \cite{K6}; $\ \widetilde{N} \ $ was
introduced and discussed originally under the name $\ N \ $ in
\cite{K2} and also in \cite{K3}, \cite{K4} and \cite{K5}).

For every pair $\ f_1, f_2 \; \colon \; X \longrightarrow Y \ $ of
maps and $\ r \; = \; 0, 1, 2, \ldots , \infty \ $ the Nielsen
number $\ N_r(f_1, f_2) \; = \; N_r(f_2, f_1) \; \in \; \Z \ $
depends only on the homotopy classes of $\ f_1 \ $ and $\ f_2 \, $.
It is extracted from the bordism class
\begin{equation}\label{equ:1.6}
  \omega_r (f_1, f_2) \; = \; \left[ \left( i_r \; \colon \; C(f_1', f_2') \; \subset \; X \times \R^r,\  \widetilde{g},\  \bar{g}_r \ \right) \right]
\end{equation}
which captures the geometric coincidence data of a generic pair $\ (
f_1', f_2' ) \ $ homotopic to $\ ( f_1, f_2 ) \, $: the vector
bundle isomorphism $\ \bar{g}_r \ $ describes the normal bundle of
the coincidence set $\ C(f_1', f_2') \, $, considered as a
submanifold of $\ X \times \R^r $, and $\ \widetilde{g} \ $ is a
canonical map into a certain pathspace $\ E(f_1, f_2) \ $ (the
so--called homotopy coincidence set). The decomposition of $\ E(f_1,
f_2) \ $ into its pathcomponents induces the {\itshape Nielsen
decomposition}
\begin{equation}\label{equ:1.7}
  C(f_1', f_2') \; = \; \underset{ A \; \in \; \pi_0(E(f_1, f_2)\,) }{ \nonComm } \ \widetilde{g}^{-1}(A)
\end{equation}
and the Nielsen number $\ N_r(f_1, f_2) \ $ counts those
pathcomponents which contribute nontrivially to $\ \omega_r(f_1,
f_2) \, $. (For more details compare \cite{K3}, \cite{K6}, as well
as sections \ref{sec:3} and \ref{sec:4}
below.)\\

In the setting of fixed point theory our Nielsen numbers $\ N_r$
(cf. \ref{equ:1.5}) all coincide with the classical Nielsen number.
However, in general they can be quite distinct: often they get
weaker but also more easily computable as $\ r \ $
increases.\\

The following {\bfseries {\itshape 2--step program}} suggests
itself.
\begin{enumerate}[label=\Roman*.]
  \item\label{TwoStepProgram:step1} {\bfseries {\itshape Decide when
    $\ \boldsymbol{\MCC(f_1, f_2)} $
    is equal to a Nielsen number
    $\ \boldsymbol{N_r(f_1, f_2)} \ $
    (and for which
    $\ \boldsymbol{r} $). }}
    In topological fixed point theory this was the central unsolved problem for nearly 60 years.
    In general coincidence theory complete answers have been given only in some simple settings; often they involve deep notions of differential topology and homotopy theory such as e.g. Kervaire invariants, all versions of Hopf invariants or the elements in the stable homotopy of spheres defined by invariantly framed Lie groups (cf. e.g. \cite{K6}, \cite{KR}).
  \item\label{TwoStepProgram:step2} {\bfseries {\itshape Determine the Nielsen numbers
    $\ \boldsymbol{N_r(f_1, f_2)} \, $.}}
\end{enumerate}

In this paper we concentrate our attention mainly on this step II 
and on the case when the domain of the maps $\ f_1, f_2 \ $ is a
sphere. Again all types of Hopf invariants (\`{a} la Ganea, Hilton,
James, \ldots) turn out to play an important r\^ole.
\providecommand{\KoschScript}[1]{\ensuremath{\mathit{#1}}}

We need some preliminary explanations. Choose an oriented compact
$n$--dimensional ball $\ B\ $ (with boundary sphere $\ \partial B \,
$), embedded in the universal covering space $\ \widetilde{Y} \ $ of
$\ Y \, $. Let
\begin{equation}\label{equ:1.8}
  b \; \colon \; \widetilde{Y} \; \longrightarrow \; \widetilde{Y} \diagup \left( \widetilde{Y} - \mathring{B} \right) \; = \; B \diagup \partial B \; \cong \; S^n
\end{equation}
denote the collapsing map. Moreover let
\begin{equation}\label{equ:1.9}
  H_{\KoschScript{C}} \; \colon \; \left[ S^m, \widetilde{Y} \right] \; \cong \; \pi_m(\widetilde{Y}) \; \longrightarrow \; \pi_m( S^n \flat \widetilde{Y} )
\end{equation}
be the Hopf--Ganea invariant homomorphism based on the cofibration
\begin{equation*}
  \KoschScript{C} \ \; \colon \ \; \partial B \ \; \subset \ \; \widetilde{Y} \setminus \mathring{B} \; \longrightarrow \; \widetilde{Y}
\end{equation*}
(cf. \cite{CLOT}, 6.44 and 6.45); here $\ S^n \flat \widetilde{Y} \
$ denotes the homotopy fiber of the inclusion of the one--point
union $\ S^n \vee \widetilde{Y} \ $ into $\ S^n \times \widetilde{Y}
\ $ (cf. \cite{G} (9) and \cite{CLOT}, §6.7).

In section \ref{sec:2} below we will present and use an explicit
geometric description of partial suspension homomorphisms
\begin{equation}\label{equ:1.10}
  e^r \; \colon \; \pi_m(S^n \flat \widetilde{Y}) \; \longrightarrow \; \pi_{m+r}(S^{n+r} \flat \widetilde{Y}) \, ,\quad r= 0, 1, 2, \ldots, \infty \, ,
\end{equation}
very closely related to those discussed by H. J. Baues (compare
\cite{Ba}, chapter 3).

\begin{thm}\label{thm:1.11}
  Let
  $\ Y \ $
  be a connected smooth $n$--dimensional manifold without boundary and write
  $\ k \; := \; \# \pi_1(Y) \ $
  for the order of the fundamental group.
  Also let
  $\ 0 \; \leq \; r \; \leq \; \infty \ $
  and assume
  $\ m \; \geq \; 2 \, $.
  Given
  $\ [f_1], [f_2] \; \in \; \pi_m(Y, y_0) \, $,
  let
  $\ [\widetilde{f}_1], [\widetilde{f}_2] \; \in \; \pi_m(\widetilde{Y}, \widetilde{y}_0) \  $
  be liftings to the universal covering space
  $\ \widetilde{Y} \ $
  of
  $\ Y \, $.
  \providecommand{\hochdot}{\boldsymbol{\cdotp}}
  \begin{enumerate}[label=,leftmargin=0cm]
    \item\label{thm:1.11:Case1} {\bfseries Case 1:
      $\ \boldsymbol{\pi_1(Y)} \ $
      is infinite or
      $\ \boldsymbol{Y} \ $
      is not compact or
      $\ \boldsymbol{m < n} \, $.}
      Then
      $\ \MCC(f_1, f_2) \; = \; 0 \ $
      and all Nielsen numbers vanish.
    \item\label{thm:1.11:Case2} {\bfseries Case 2:
      $\ \boldsymbol{2 \leq k < \infty} \, $.}
      Choose a map
      $\ a^{\hochdot} \; \colon \; (\widetilde{Y}, \widetilde{y}_0 ) \; \longrightarrow \; (\widetilde{Y},\widetilde{y}_0 ) \ $
      which is freely homotopic to a fixed point free selfmap
      $\ a \ $
      of
      $\ \widetilde{Y} \ $
      (e.g. to a covering transformation
      $\ a\, \neq \, $
      identity map).
      Then precisely one of the following four conditions holds (compare \ref{equ:1.8} and \ref{equ:1.9}):
      \begin{alignat*}{3}
    (*_k)\quad H_{\KoschScript{C}}(\widetilde{f}_1) &\;\neq\; H_{\KoschScript{C}}(\widetilde{f}_2)      &       &\text{ or }        &       b &\comp \widetilde{f}_1 \;\not\sim\; b \comp \widetilde{f}_2 \;\not\sim\; b \comp a^{\hochdot} \comp \widetilde{f}_1 \, ; \\
    (*_{k-1})\quad H_{\KoschScript{C}}(\widetilde{f}_1) &\;=\; H_{\KoschScript{C}}(\widetilde{f}_2)     &       &\text{ and }       &       b &\comp \widetilde{f}_1 \;\not\sim\; b \comp a^{\hochdot} \comp \widetilde{f}_1 \;\sim\; b \comp \widetilde{f}_2 \, ; \\
    (*_1)\quad H_{\KoschScript{C}}(\widetilde{f}_1) &\;=\; H_{\KoschScript{C}}(\widetilde{f}_2)     &       &\text{ and }       &       b &\comp a^{\hochdot} \comp \widetilde{f}_1 \;\not\sim\; b \comp \widetilde{f}_1 \;\sim\; b \comp \widetilde{f}_2 \, ; \\
    (*_0)\quad H_{\KoschScript{C}}(\widetilde{f}_1) &\;=\; H_{\KoschScript{C}}(\widetilde{f}_2)\quad        &       & \text{ and }\quad     &       b &\comp a^{\hochdot} \comp \widetilde{f}_1 \;\sim\; b \comp \widetilde{f}_1 \;\sim\; b \comp \widetilde{f}_2 \, .
      \end{alignat*}
      If the condition
      $\ (*_i) \ $
      is satisfied for
      $\ i \; = \; 0, 1, k-1 \text{ or } k \, $,
      then
      \begin{equation*}
    N^{\#}(f_1, f_2) \; = \; i \, .
      \end{equation*}
      Precisely the analogous result holds for
      $\ N_r(f_1, f_2) \, ,\ r \;=\; 0, 1, 2, \ldots, \infty \, $,
      when we replace
      $\ H_{\KoschScript{C}} \ $
      by
      $\ e^r \comp H_{\KoschScript{C}} \ $
      (cf. \ref{equ:1.10}) and
      $\ b \comp \widetilde{f} \ $
      by the $r$--fold (standard) suspension
      $\ E^r([ b \comp \widetilde{f} ]) \; \in \; \pi_{m+r}(S^{n+r}) \ $
      (for
      $\ \widetilde{f} \; = \; \widetilde{f}_i \, , \ a^{\hochdot} \comp \widetilde{f}_i \, , \ i= 1, 2 \, $).
    \item\label{thm:1.11:Case3} {\bfseries Case 3:
    $\ \boldsymbol{Y} \ $
    is simply connected and admits a fixed point free map
    $\ \boldsymbol{a} \,$.}
    Deform
    $\ a \ $
    to a base point preserving map
    $\ a^{\hochdot} \; \colon \; (Y, y_0) \; \longrightarrow \; (Y, y_0) \, $.
    Then
    \begin{equation*}
      N_r(f_1, f_2) \, = \,
      \begin{cases}
        1 &  \text{if } e^r \left( H_{\KoschScript{C}}(f_1) \right) \neq e^r \left( H_{\KoschScript{C}}(a^{\hochdot} \comp f_2) \right) \text{ or } E^r\left( \left[ b \comp f_1 \right] \right) \neq  E^r\left( \left[ b \comp a^{\hochdot} \comp f_2 \right] \right)  ; \\
        0 &  \text{otherwise.}
      \end{cases}
    \end{equation*}
  \end{enumerate}
\end{thm}

In particular, the values which our Nielsen numbers may possibly
assume are severely restricted. In fact, only two or at most three
different values can occur:

\begin{prop}\label{prop:1.12}
  Let
  $\ 0 \leq r \leq \infty, \ m \geq 2 ,\ Y $
  and
  $\ k \; := \; \# \pi_1(Y) $,
  as well as
  $\ [f_1], [f_2] \; \in \; \pi_m(Y) \ $
  be as in theorem \ref{thm:1.11}.

  Then
  $\ N_r(f_1, f_2) \; \in \; \lbrace 0, 1, k \rbrace $.
  Furthermore, if
  $\ N_r(f_1, f_2) \; \notin \; \lbrace 0, k \rbrace $,
  then the following restrictions must all be satisfied:
  \begin{enumerate}[label=(\roman*)]
    \item
      $ n \ $
      is even and
      $\ m \geq n \geq 4 $,
      or else
      $\ m = 2 \ $
      and
      $\ Y = \RP(2) $; \hspace{3mm} and
    \item the manifold
      $\ Y \ $
      is closed, not orientable and not a product of two manifolds with strictly positive dimensions;
      $\ \pi_1(Y) \cong \Z_2; \; \chi(Y) \neq 0 $;
      $\ Y \ $
      admits no fixed point free selfmap.
      Also the homomorphism
      $\ i_* \; \colon \; \pi_m(Y \setminus \lbrace * \rbrace ) \; \longrightarrow \; \pi_m(Y) \ $
      (induced by the inclusion of
      $\ Y $,
      punctured at some point
      $\ * $)
      is not surjective.
      Moreover, the composed homomorphism
      \begin{equation*}
        E \comp \partial_{Y} \; \colon \; \pi_m(Y) \overset{\partial_{Y}}{\longrightarrow} \pi_{m-1}( S^{n-1} ) \overset{E}{\longrightarrow} \pi_m( S^n )
      \end{equation*}
      is nontrivial; here
      $\ \partial_{Y} \ $
      denotes the boundary homomorphism in the homotopy sequence of the tangent sphere bundle
      $\ \ST(Y) \ $
      over
      $\ Y $.
  \end{enumerate}
\end{prop}

A typical example of a manifold $\ Y \ $ which may satisfy all these
restrictions is even dimensional real projective space $\ \RP(n) =
S^n \diagup \Z_2 \cdot a \ $, the orbit space of the antipodal
involution $\ a $. More generally, let us illustrate theorem
\ref{thm:1.11} by examples where $\ Y \ $ is an arbitrary spherical
space form $\ S^n \diagup G $. Here our criteria can be expressed in
terms of the Hopf--Hilton invariant homomorphisms
\begin{equation}\label{equ:1.13}
  h_j' \; \colon \; \pi_m(S^n) \; \longrightarrow \; \pi_m(S^{n+j (n-1)})\, , \quad j= 1, 2, \ldots
\end{equation}
which correspond to the basic Whitehead products
\begin{equation}\label{equ:1.14}
  w_j' \; := \; \left[ \iota_2 , \ldots , \left[ \iota_2, \left[ \iota_1, \iota_2 \right] \ldots \right] \, \right] \; \in \; \pi_{n+j (n-1)}( S^n \vee S^n )
\end{equation}
with one factor $\ \iota_1 \ $ and $\ j \ $ factors $\ \iota_2 \ $
(cf. \cite{H} and section \ref{sec:5} below).
Define\renewcommand{\theequation}{\arabic{section}.\arabic{equation}'}\addtocounter{equation}{-2}
\begin{equation}\label{equ:1.13Strich}
  h' \; := \; (h_1', h_2', \ldots) \; \colon \; \pi_m(S^n) \; \longrightarrow \; \underset{j \geq 1}{\oplus} \pi_m(S^{n+j (n-1)}
\end{equation}\numberwithin{equation}{section}\addtocounter{equation}{1}
and $\ h \; := \; (\id, h') \, $. Thus e.g.
\begin{equation*}
  E^r \comp h \; := \; (E^r, E^r \comp h_1', E^r \comp h_2', \ldots) \, .
\end{equation*}

\begin{thm}\label{thm:1.15}
  Let
  $\ G \ $
  be a finite group acting smoothly and freely on the sphere
  $\ S^n \ $
  and let
  $\ Y \; = \; S^n \diagup G \ $
  be the resulting orbit space.
  Assume
  $\ m, n \geq 2 \ $
  and
  $\ 0 \leq r \leq \infty \, $.
  Also let
  $\ \iota \; \in \; \pi_n(S^n) \ $
  be represented by the identity map
  $\ \id \, $.

  For all homotopy classes
  $\ [f_i] \; \in \; \pi_m (S^n \diagup G, y_0) \ $
  and their liftings
  $\ [\widetilde{f}_i] \; \in \; \pi_m (S^n, \widetilde{y}_0), \; i = 1,2 \,$,
  we have:\\
  If n is odd, then:
  \begin{equation*}
    N_r(f_1, f_2) \, = \,
    \begin{cases}
      \# G & \text{if } E^r \comp h([\widetilde{f}_1]) \; \neq \; E^r \comp h([\widetilde{f}_2]) \, ; \\
      0    & \text{otherwise} \, .
    \end{cases}
  \end{equation*}
  If n is even and
  $\ G = 0\, $,
  then:
  \begin{equation*}
    N_r(f_1, f_2) \, = \,
    \begin{cases}
      1  & \text{if } E^r \comp h([\widetilde{f}_1]) \; \neq \; E^r \comp h((-\iota) \comp [\widetilde{f}_2]) \, ; \\
      0  & \text{otherwise} \, .
    \end{cases}
  \end{equation*}
  If n is even and
  $\ G \cong \Z_2 \, $,
  then:
  \begin{equation*}
    N_r(f_1, f_2) \, = \,
    \begin{cases}
      2 & \text{if } E^r \comp h'[\widetilde{f}_1] \; \neq \; E^r \comp h'[\widetilde{f}_2]
      \text{ or } E^r[\widetilde{f}_1] \; \notin \left\lbrace E^r[\widetilde{f}_2], \, E^r((-\iota) \comp [\widetilde{f}_2])
      \right\rbrace \, ; \\
      1 & \text{if } E^r \comp h'[\widetilde{f}_1] \; = \; E^r \comp h'[\widetilde{f}_2]
      \text{ and } E^r[\widetilde{f}_1] \; \in \left\lbrace E^r[\widetilde{f}_2], \, E^r((-\iota) \comp [\widetilde{f}_2]) \right\rbrace \\
    &  \qquad\qquad\qquad\qquad\qquad\quad\ \text{ but } E^r[\widetilde{f}_2] \; \neq \; E^r((-\iota) \comp [\widetilde{f}_2]) \, ; \\
      0 & \text{if } E^r \comp h'[\widetilde{f}_1] \; = \; E^r \comp h'[\widetilde{f}_2]
      \text{ and } E^r[\widetilde{f}_1] \; = \; E^r[\widetilde{f}_2] \; = \; E^r((-\iota) \comp [\widetilde{f}_2]) \, .
    \end{cases}
  \end{equation*}

  In particular, all these Nielsen numbers depend only on the order
  $\ \# G \ $
  of the group
  $\ G \ $
  and not on the $G$--action itself.
\end{thm}

Note that $\ \#G \, \leq \, 2 \ $ when $\ n \ $ is even (as seen by
a simple argument involving Euler characteristics).

\begin{cor}\label{cor:1.16}
  {\bfseries(Case
  $\ \boldsymbol{r = 0} $).} Recall that
  $\ N_0 \, \equiv \, N^{\#} \ $
  (cf. \ref{equ:1.5}).

  Let
  $\ a \; \colon \; S^n \; \longrightarrow \; S^n \ $
  denote the antipodal map.

  \begin{enumerate}[label=(\roman*)]
    \item\label{cor:1.16-item1} Suppose that
      $\ n \ $
      is odd or
      $\ G \; = \; 0 $.
      Then
      \begin{equation*}
        N_0(f_1, f_2) \; = \;
        \begin{cases}
          \# G & \quad\text{if }\; \widetilde{f}_1 \; \not\sim \; a \comp \widetilde{f}_2 \, ; \\
          0    & \quad\text{if }\; \widetilde{f}_1 \; \sim \; a \comp \widetilde{f}_2 \, .
        \end{cases}
      \end{equation*}
      In particular, if
      $\ n \, \geq \, 3\ $
      is odd then
      $\ N_0(f_1, f_2) \ $
      takes the value
      $\ 0 \ $
      or
      $\ \# G \ $
      according as
      $\ \widetilde{f}_1 \ $
      is (freely) homotopic to
      $\ \widetilde{f}_2 \ $
      or not, resp..
    \item\label{cor:1.16-item2} Suppose that
      $\ n \ $
      is even and
      $\ G \cong \Z_2 $.
      Then
      \begin{equation*}
        N_0(f_1, f_2) \; = \;
        \begin{cases}
          2 & \text{if }\; \widetilde{f}_1 \; \not\sim \; \widetilde{f}_2 \text{ and } \widetilde{f}_1 \; \not\sim \; a \comp \widetilde{f}_2 \, ; \\
          1 & \text{if }\; \widetilde{f}_1 \; \sim \; \widetilde{f}_2 \text{ or } \widetilde{f}_1 \; \sim \; a \comp \widetilde{f}_2, \text{  but } \widetilde{f}_2 \; \not\sim \; a \comp \widetilde{f}_2 \, ; \\
          0 & \text{if }\; \widetilde{f}_1 \; \sim \; \widetilde{f}_2 \; \sim \; a \comp \widetilde{f}_2 \, .
        \end{cases}
      \end{equation*}
  \end{enumerate}
\end{cor}

\begin{cor}\label{cor:1.17}
  {\bfseries(Case
  $\ \boldsymbol{r \geq 1} $).}

  Here
  $\ E^r((-\iota) \comp [\widetilde{f}_2]) \; = \; - E^r([\widetilde{f}_2]) $.
  We get e.g. for
  $\ n \ $
  even,
  $\ G \cong \Z_2 $:
  \begin{equation*}
    N_r(f_1, f_2) \; = \;
      \begin{cases}
        2 & \text{if } E^r \comp h'(\widetilde{f}_1) \; \neq \; E^r \comp h'(\widetilde{f}_2) \text{ or } E^r(\widetilde{f}_1) \; \neq \; \pm E^r(\widetilde{f}_2) \,; \\
        1 & \text{if } E^r \comp h'(\widetilde{f}_1) \; = \; E^r \comp h'(\widetilde{f}_2) \text{ and } E^r(\widetilde{f}_1) \; = \; \pm E^r(\widetilde{f}_2) \text{ has order } > 2 \,; \\
        0 & \text{if } E^r \comp h'(\widetilde{f}_1) \; = \; E^r \comp h'(\widetilde{f}_2) \text{ and } E^r(\widetilde{f}_1) \; = \; \pm E^r(\widetilde{f}_2) \text{ has order } \leq 2 \,.
      \end{cases}
  \end{equation*}
\end{cor}

Note that Hopf invariantes play no r\^ole here in our criteria for
the basic Nielsen number $\ N^{\#} \; = \; N_0 $. However they are
often decisive when $\ r \; \geq \; 1$.

\begin{exa}\label{exa:1.18}
  $ m\; = \; 3, \; n \; = \; 2 \ $
  and
  $\ G \; = \; 0$.

  For all maps
  $\ f_1, f_2 \; \colon \; S^3 \; \longrightarrow \; S^2 \ $
  and
  $\ 0 \; \leq \; r \; \leq \; \infty \ $
  we have:
  $\ N_r(f_1, f_2) \ $
  equals
  $\ 0 \ $
  or
  $\ 1 $,
  resp., according as
  $\ f_1 \; \sim \; f_2 \ $
  or
  $\ f_1 \; \not\sim \; f_2 $,
  resp.
  When
  $\ r \;\geq\; 1 \ $
  this is detected only by
  $\ E^r h' \; = \; E^r h'_1 \ $
  or, equivalently, by the classical Hopf invariant
  $\ H $.
  Indeed,
  \begin{equation*}
    h'(f_i) \; = \; \pm H(f_i) \comp \iota \; \in \; \pi_3(S^3) \cong \Z, \ \ i = 1,2 \, ,
  \end{equation*}
  (cf. \cite{W}, XI, 8.17)
  persists under all iterated suspensions; in contrast,
  $\ E^r(f_i) \ $
  measures only
  $\mod{2} \ $
  values in
  $\ \pi_{r+3}(S^{r+2}) \; \cong \; \Z_2 \ $
  when
  $\ r \; \geq \; 1 $.
  \qed
\end{exa}

\begin{cor}\label{cor:1.19}
  {\bfseries(Case
  $\ \boldsymbol{r = \infty} $).}

  For all
  $\ [f_1], [f_2] \; \in \; \pi_m(S^n), \, m,n \geq 2 $,
  we have:
  \begin{equation*}
    N_{\infty}(f_1, f_2) \; = \;
    \begin{cases}
      0 & \text{if } E^{\infty} \comp \gamma_j([f_1]) \; = \; E^{\infty} \comp \gamma_j(((-1)^{n+1} \cdot \iota) \comp [f_2]) \ \text{ for all } \ j \geq 1 \, ; \\
      1 & \text{otherwise.}
    \end{cases}
  \end{equation*}
  Here
  $\ E^{\infty} \comp \gamma_j \; \colon \; \pi_m(S^n) \; \longrightarrow \pi^S_{m-1-j(n-1)} \ $
  is defined by the infinitely suspended Hopf--James invariant,
  $\ j = 1,2, \ldots \; $.
\end{cor}

This was proved already in \cite{K2}, §8, by interpreting $\
E^{\infty} \comp \gamma_j \ $ via ($j-1$)--tuple selfintersections
of framed immersions.

\begin{cor}\label{cor:1.20}
  Let
  $\ Y = S^n \diagup G \ $
  and
  $\ m, n \geq 2 \ $
  be as in theorem \ref{thm:1.15}.
  Assume that
  $\ 2 \alpha \; = \; 0 \ $
  for all
  $\ \alpha \; \in \; \pi^S_{m-n} \ $
  (according to the tables in \cite{T} this holds e.g. when
  $\ m-n \, = \, 1, 2, 4, 5, 6, 8, 9, 12, 14, 16 \text{ or } 17 $).
  Then
  $\ N_{\infty}(f_1, f_2) \; \in \; \lbrace 0, \, \# G \rbrace \ $
  for all
  $\ [f_1], [f_2] \; \in \; \pi_m(S^n \diagup G) $.
\end{cor}

Hopf--James invariants occur not only in the criteria which
determine Nielsen numbers (as e.g in corollary \ref{cor:1.19}).
They play also an important r\^ole --- via EHP--sequences --- in the computations needed to decide whether these criteria are fulfilled.\\

We are particularly interested on those settings in theorem
\ref{thm:1.15} where the Nielsen numbers can possibly assume three
distinct values (compare proposition \ref{prop:1.12}). Thus let $\ n
\, \geq \, 2 \ $ be even and $\ G \, \cong \, \Z_2 $. In several
important cases we are able to carry out our 2--step program and
compute all Nielsen numbers as well as the minimum number $\MCC$.

\begin{exa}\label{exa:1.21}
  {\bfseries$ \boldsymbol{\ m\; \leq \; 2n - 1; \; n \; \geq \; 2 \ \text{even} } $.}

  Given maps
  $\ f_1, f_2 \, \colon \, S^m \longrightarrow S^n \diagup \Z_2 \,$,
  all Nielsen numbers agree,
  $ \ N_r(f_1, f_2) \, = \, N^{\#}(f_1, f_2) \ $
  for
  $\ 0 \, \leq \, r \, \leq \, \infty \, $,
  and are determined by corollary \ref{cor:1.16}\ref{cor:1.16-item2}
  or, equivalently, by corollary \ref{cor:1.17}.

  If
  $\ m \, \leq \, 2n -3 \, $,
  then:
  \begin{equation*}
    \MCC(f_1, f_2) \, = \, N^{\#}(f_1, f_2) \, .
  \end{equation*}

  If
  $\ m \, = \, 2n - 2 \, $,
  then:
  $\ \MCC(f_1, f_2) \, = \, N^{\#}(f_1, f_2) \ $
  {\bfseries except} precisely if
  $\ n \, = \, 16, 32, 64 \ $
  (or maybe
  $\ 128 \ $)
  and
  $\ f_1 \, \sim \, f_2 \, =: \, f \ $,
  and
  $\ N^{\#}(f, f) \, = \, 0 \,$,
  but the lifting
  $\ \widetilde{f} \, \colon \, S^{2n - 2} \, \longrightarrow \, S^n \ $
  of
  $\ f\ $
  has a nontrivial {\bfseries Kervaire invariant}
  $\ K(\widetilde{f}) \, = \, 1 \, $.

  If
  $\ m \, = \, 2n - 1 \, $,
  then
  $\ \MCC(f_1, f_2) \, = \, N^{\#}(f_1, f_2) \ $
  {\bfseries except} precisely if
  $\ n \, \equiv \, 2(4), \; n \, \geq \, 6 \,$,
  and
  $\ f_1 \, \sim \, f_2 \, =: \, f \ $
  and
  $\ N^{\#}(f, f) \, = \, 0 \,$
  and the {\bfseries Hopf invariant} of the lifting
  $\ \widetilde{f} \, \colon \, S^{2n - 1} \, \longrightarrow \, S^n \ $
  of
  $\ f\ $
  is not divisible by
  $\ 4 \, $.
  \qed
\end{exa}

\begin{exa}\label{exa:1.22}
  {\bfseries$ \boldsymbol{\ m\; \leq \; n + 3; \; n \; \geq \; 2 \ \text{even} } $.}

  Given maps
  $\ f_1, f_2 \, \colon \, S^m \, \longrightarrow\, S^n \diagup \Z_2 \ $
  and
  $\ 0 \, \leq \, r \, \leq \, \infty \, $,
  we have
  \begin{equation*}
    \MCC(f_1, f_2) \, = N_0(f_1, f_2) \, = \, \cdots \, = \, N_r(f_1, f_2) \, = \, \cdots \, = \, N_{\infty}(f_1, f_2) \, .
  \end{equation*}
  I.e. the Nielsen numbers do not depend on
  $\ r \ $
  and agree with the minimum number
  $\ \MCC \,$;
  they are determined by corollary \ref{cor:1.16}\ref{cor:1.16-item2}
  or, equivalently, by corollary \ref{cor:1.17}.
  \qed
\end{exa}
\vspace{5mm}

When the domain and target manifolds of $\ f_1, f_2 \ $ have the
same dimension $\ m \, = \, n \ $ (e.g. in fixed point theory) and
also in examples \ref{exa:1.21} and \ref{exa:1.22}\, our Nielsen
numbers are independent of $\ r $. However, they can be quite
distinct in general (and loose strength, but gain in computability
as $\ r \ $ increases).

Given $\ Y, \; r, \; m \ $ and $\ X \, = \, S^m \ $ as in theorem
\ref{thm:1.11}, a standard stability argument shows that $\ N_r \,
\equiv \,
  N_{r+1} \, \equiv \,
  \ldots \, \equiv \,
  N_{\infty} \quad \text{ for }\ r \, \geq \, m - 2n +2 \, $;
thus the number of possibly different Nielsen number functions $\
N_r\ $ is limited by the so--called "degree of instability" $\ m -2n
+3 \ $ (cf. \cite{KR}, 1.12). For distinguishing them we define
\begin{equation}\label{equ:1.23}
  \#^i_r(m, Y) \, := \, \# \left\lbrace
    \left( [f_1], [f_2] \right) \, \in \,
    \left( \pi_m(Y) \right)^2 \; \vert \; N_r(f_1,f_2)
  \, = \, i \, \right\rbrace
  \  \in \  \left\lbrace 0, 1, 2, \ldots, \infty \right\rbrace
\end{equation}
for $\ i \, = \, 0, 1, \ldots \ $. These cardinalities sum up to the
square of $\ \# \pi_m(Y) \ $ and vanish when $\ i \, \notin \,
\left\lbrace \, 0,\, 1,\, k \, := \, \# \pi_1(Y) \right\rbrace \ $
(cf. proposition \ref{prop:1.12}). Clearly, $\ \#^0_r(m, Y) \, \leq
\, \#^0_{r+1}(m,Y)\ $ and $\ \#^k_r(m, Y) \, \geq \, \#^k_{r+1}(m,
Y) \, $.

\begin{prop}\label{prop:1.24}
  Let
  $\ m, \; Y \, = \, S^n \diagup G \ $
  and
  $\ r \ $
  be as in theorem \ref{thm:1.15}.
  Then the following conditions are equivalent:
  \begin{enumerate}[label=(\roman*)]
    \item\label{prop:1.24:1}
      $ N_r \, \equiv \, N_{r+1} \ $
      (i.e.
      $\ N_r(f_1, f_2) \, = \, N_{r+1}(f_1,f_2) \ $
      for all
      $\ f_1, f_2 \, \colon \, S^m \longrightarrow Y \, $) ;
    \item\label{prop:1.24:2}
      $ \#\Ker( E^r \comp h ) \, = \, \#\Ker( E^{r+1} \comp h ) \ $ (cf. \ref{equ:1.13}ff);
    \item\label{prop:1.24:3}
      $ \#^0_r(m, Y) \, = \, \#^0_{r+1}(m, Y) \, $;
    \item\label{prop:1.24:4}
      $ \#^i_r(m, Y) \, = \, \#^i_{r+1}(m, Y) \, $
      for all
      $\ i \, = \, 0, 1, \ldots \, $;
  \end{enumerate}
\end{prop}

In particular, when comparing Nielsen number functions $\ N_r,\,
N_{r'},\, 0 \ \leq \ r,\, r' \ \leq \ \infty \, $, it suffices to
count how often they vanish.

\begin{exa}\label{exa:1.25}
  {\bfseries $ \boldsymbol{m \, = \, 16, \; Y \, = \, S^6 \diagup \Z_2 }. $}
  Here the Nielsen numbers
  $\ N_r \ $
  determine five distinct functions on pairs
  $\ f_1, f_2 \, \colon \, S^{16} \, \longrightarrow \,S^6 \diagup \Z_2 \ $
  of maps:
  \begin{equation*}
    \MC \, \equiv \,
    \MCC \, \equiv \,
      N_0 \, \not\equiv \,
      N_1 \, \equiv \,
      N_2 \, \not\equiv \,
      N_3 \, \not\equiv \,
      N_4 \, \not\equiv \,
      N_5 \, \equiv \,
      N_6 \, \equiv \,
      N_{\infty}
  \end{equation*}
  (stability arguments would not allow more than seven distinct such functions anyway).
  The precise value distributions are given by Table \ref{tab:1.26}.

  Moreover there are precisely four {\itshape"loose"} pairs
  $\ ([f_1], [f_2]) \in \pi_{16}(S^6 \diagup \Z_2 )^2 \ $
  (i.e.
  $\ \MCC(f_1, f_2) \, = \, 0 \ $
  or, equivalently,
  $\ f_1 \ $
  and
  $\ f_2 \ $
  can be deformed away from one another);
  they have the form
  $\ ([f_1], [f_2]) \, = \, ([p \comp \widetilde{f}], [p \comp \widetilde{f}]) \ $
  where
  $\ p \ $
  denotes the projection and
  $\ [\widetilde{f}] \ $
  lies in the subgroup
  \begin{equation*}
    \Z_2(4 \nu_6 \comp \sigma_9) \, \oplus \, \Z_2(\eta_6 \comp \mu_7) \ \ \subset \ \ \pi_{16}(S^6) \, \cong \, \Z_8 \, \oplus \, \Z_2 \, \oplus \, \Z_9
  \end{equation*}
  (compare \cite{T}, theorem 7.3).
  \qed
\end{exa}

\addtocounter{table}{25}
\begin{table}[H]
  \centering
  \begin{tabular}{c|c|c|c|c|c}
    $ r $ & $ 0 $ & $ 1, 2 $ & $ 3 $ & $ 4 $ & $ r \geq 5 $\\ \toprule
    $ \#\Ker(E^r \comp h) $ & 1 & 3 & 6 & 12 & 24\\ \midrule
    $ \#_r^0(16, S^6 \diagup \Z_2) $ & 4 & 36 & 144 & 576 & 1152\\ \hline
    $ \#_r^1(16, S^6 \diagup \Z_2) $ & 280 & 792 & 1440 & 2304 & 4608\\ \hline
    $ \#_r^2(16, S^6 \diagup \Z_2) $ & 20452 & 19908 & 19152 & 17856 & 14976\\\midrule
    $ \#_r^0(16, S^6) $ & 144 & 432 & 864 & 1728 & 3456\\ \hline
  \end{tabular}
\caption{The value distributions of the Nielsen numbers $ \ N_r \ $
for pairs of maps from $\ S^{16} \ $ to $\ S^6 \diagup \Z_2 \ $ and
to $\ S^6\, $. Here the suspension $ E^r \comp h \,  = \, (E^r, 0 )
\, \colon \, \pi_{16}(S^6) \, \longrightarrow \, \pi_{r+16}(S^{r+6})
$ is equally relevant in both cases.} \label{tab:1.26}
\end{table}\addtocounter{equation}{1}

\begin{exa}\label{exa:1.27}
  {\bfseries $ \boldsymbol{m \, = \, 16, \; Y \, = \, S^6 }. $}
  Here there are again five distinct Nielsen numbers (as in example \ref{exa:1.25}):
  \begin{equation*}
    \MC \, \equiv \,
    \MCC \, \equiv \,
      N_0 \, ; \
      N_1 \, \equiv \
      N_2 \, ; \
      N_3 \, ; \
      N_4 \  \text{ and }\
      N_5 \, \equiv \, \ldots \, \equiv \,
      N_{\infty} \, .
  \end{equation*}
  They take only the values
  $\ i \, = \, 0 \ $
  and
  $\ 1 \, $.
  The precise value distribution is given in Table \ref{tab:1.26}.
  (Note that
  \begin{equation*}
    \#^0_r(16, S^6) \, + \, \#^1_r(16, S^6) \, = \, (\# \pi_{16} (S^6) )^2 \, = \, 20736 \,).
  \end{equation*}

  Moreover precisely the
  $\ 144 \ $
  pairs of the form
  $\ (\, [\widetilde{f}]\, ,\, - [\widetilde{f}]\,), \  [\widetilde{f}] \, \in \, \pi_{16}(S^6) \,$
  , are loose;
  unless
  $\ 2 [\widetilde{f}] \, = \, 0 \ $
  they do not project to loose pairs in
  $\ S^6 \diagup \Z_2 \ $
  (compare example \ref{exa:1.25}).
  \qed
\end{exa}

\vspace{1ex} Finally let us come back to the central objects of
study in topological coincidence theory: the minimum numbers MC and
MCC of coincidence points and of coincidence pathcomponents, resp.
(cf. \ref{equ:1.2}, \ref{equ:1.3}, and compare \cite{B}, p.9). What
can we say about them for maps into spherical space forms, once the
Nielsen numbers -- or at least $\ N_0 \ $ -- are understood?

\begin{thm}\label{thm:1.28}
  Let the finite group
  $\ G \ $
  act smoothly and freely on
  $\ S^n \ $
  and consider (basepoint preserving) maps
  \begin{equation*}
    f_1, f_2 \; \colon \; S^m \, \longrightarrow \, Y \, = \, S^n \diagup G, \ \ m, n \, \geq \, 1 \, .
  \end{equation*}

  When
  $\ Y \, \cong \, S^n \ $
  then
  $\ \MCC \, \equiv \, N_0 \ $
  (i.e.,
  $\ \MCC(f_1, f_2) \, = \, N_0(f_1, f_2) \ $
  for all
  $\ f_1, f_2 \, $).

  When
  $\ \#G \, \geq \, 2 \ $
  then
  $\ \MCC \, \equiv \, N_0 \ $
  if and only if the {\upshape'Wecken condition'}
  \begin{equation}\label{equ:1.29}
    0 \, = \, \partial \lt \pi_m \lt S^n \rt \rt \, \cap \, \Ker \lt E \, \colon \, \pi_{m-1} \lt S^{n-1} \rt \, \longrightarrow \, \pi_m \lt S^n \rt \rt
  \end{equation}
  holds; here
  $\ \partial \, := \, \partial_{S^n} \ $
  denotes the boundary homomorphism in the exact homotopy sequence of the tangent sphere bundle
  $\ \ST(S^n) \ $
  fibred over
  $\ S^n \ $
  (as in proposition \ref{prop:1.12}).
\end{thm}

Clearly condition \ref{equ:1.29} is satisfied when $\ n \ $ is odd
or $\ n \, = \, 2 \ $ or $\ m \, < \, n\ $ or in the "stable range"
$\ m \, < \, 2n - 2 \, $. But it can already fail to hold when $\
m\, = \, 2n - 2 \ $ or $\ m\, = \, 2n - 1 \, $. This explains the
appearance of the Kervaire invariant and of the $\ \mod{4}\ $ Hopf
invariant in the criteria in example \ref{exa:1.21}. For information
concerning the next six nonstable dimension settings see \cite{KR}.
For the many geometric consequences of (a possible failure of)
condition \ref{equ:1.29} see e.g. \cite{K6}, corollary 1.21.

Once the minimum number $\ \MCC\ $ of coincidence {\itshape
pathcomponents} is determined what about the minimum number $\ \MC\
$ of coincidence {\itshape points}? The answer for spherical space
forms involves Hopf invariants in a decisive way.

\begin{thm}\label{thm:1.30}
  Given maps
  $\ f_1, f_2 \; \colon\; S^m \, \longrightarrow \, Y \, = \, S^n \diagup G\, , \; m, n \, \geq\, 1\, $,
  as in theorem \ref{thm:1.28}, we have:
  \begin{enumerate}[label=(\roman*)]
    \item\label{thm:1.30-item1} If
      $\ \MC(f_1, f_2) \, < \, \infty \,$,
      then
      $\ \MC(f_1, f_2) \, = \, \MCC(f_1, f_2) \,$.
    \item\label{thm:1.30-item2} If
      $\ n \, = \, 1 \ $
      or
      $\ m \, < \, n \,$,
      then
      $\ \MC(f_1, f_2) \, = \, \MCC(f_1, f_2) \, < \, \infty \,$.
  \end{enumerate}
  If
  $\ m, n \, \geq \, 2 \,$,
  then
  \begin{equation*}
    \MC(f_1, f_2) \, < \, \infty \; \Longleftrightarrow \; [\widetilde{f}] \, \in \,
    \begin{cases}
      E(\Ker(\underline{h}'))   & \text{if } \#G \, \geq \, 3 \,; \\
      E(\pi_{m-1}(S^{n-1})) & \text{if } \#G \, \leq \, 2 \,. \\
    \end{cases}
  \end{equation*}
  Here
  $\ [\widetilde{f}] \; := \; [\widetilde{f}_1] - [\widetilde{f}_2] \, \in \, \pi_m(S^n) \ $
  where
  $\ [\widetilde{f}_i] \ $
  is obtained by lifting
  $\ [f_i] \, \in \, \pi_m(Y,y_0) , \; i\,=\,1,2 \,$;
  \renewcommand{\theequation}{\arabic{section}.\arabic{equation}'}\addtocounter{equation}{-1}
  \begin{equation}\label{equ:1.30Strich}
  \underline{h}' \; \colon \; \pi_{m-1}(S^{n-1}) \; \longrightarrow \; \pi_{m-1}(S^{2n-3}) \oplus \pi_{m-1}(S^{3n-5})^2 \oplus \pi_{m-1}(S^{4n-7})^3 \oplus \ldots
  \end{equation}
  \numberwithin{equation}{section}is the {\upshape total} Hopf--Hilton homomorphism (which involves {\upshape all} basic Whitehead products and not just a selection as in (\ref{equ:1.13Strich}); note also the different dimensions here).
\end{thm}

Since Hopf invariants vanish on suspended maps, we have the
inclusions
\begin{equation}
  E^2(\pi_{m-2}(S^{n-2})) \, \subset \, E(\Ker(\underline{h}')) \, \subset \, E(\pi_{m-1}(S^{n-1})) \, \subset \, \Ker h' \, \subset \, \pi_m(S^n)
\end{equation}
at least when $\ n \, \geq \, 3 \ $ (cf. \ref{equ:1.13Strich} and
\ref{equ:1.30Strich}).

\begin{cor}\label{cor:1.32}
  Assume
  $\ n \, \geq\, 2 \,$.
  \begin{enumerate}[label=(\roman*)]
    \item\label{cor:1.32-item1} If
      $\ \MC(f_1, f_2) \, < \, \infty \ $
      then
      $\ h'([\widetilde{f}_1]) \, = \, h'([\widetilde{f}_2]) \,$.

      If
      $\ n \ $
      is even and
      $\ m \, \leq \, 3n - 4 \,$,
      then:
      \begin{equation*}
        \MC(f_1, f_2) \, < \, \infty \; \Longleftrightarrow \; h'([\widetilde{f}_1]) \, = \, h'([\widetilde{f}_2]) \,.
      \end{equation*}
      \item\label{cor:1.32-item2} Now assume that
    $\ n \ $
    is odd. Then:
    \begin{equation*}
      \MC(f_1, f_2) \, = \, \MCC(f_1, f_2) \, = \,
      \begin{cases}
        \#G     & \text{if } \widetilde{f}_1 \, \not\sim \, \widetilde{f}_2 \,; \\
        0       & \text{if } \widetilde{f}_1 \, \sim \, \widetilde{f}_2 \,; \\
      \end{cases}
    \end{equation*}
    {\bfseries provided}
    $\ [\widetilde{f}_1] - [\widetilde{f}_2] \ $
    lies in
    $\ E(\pi_{m-1}(S^{n-1})) \ $
    when
    $\ \#G \, \leq \, 2 \ $
    or in
    $\ E(\Ker(\underline{h}')) \ $
    when
    $\ \#G \, \geq \, 3 \,$.
    If this condition fails to hold, then
    $\ \MCC(f_1, f_2) \, = \, \#G \ $
    but
    $\ \MC(f_1, f_2) \ $
    is infinite.
  \end{enumerate}
\end{cor}

Thus if $\ n \ $ is odd and $\ E(\Ker(\underline{h}')) \, \neq \,
E(\pi_{m-1}(S^{n-1})) \ $ the finiteness of $\ \MC(f_1, f_2) \ $
depends strongly on $\ \#G \,$.

\begin{exa}\label{exa:1.33}
  Let
  $\ m \, = \, 2n - 2 \ $
  and
  $\ n \, = \, 3, \, 5 \ $
  or
  $\ 9 \,$,
  resp..
  Then
  $\ \pi_m(S^n) \ $
  is a cyclic group of order
  $\ 2, \ 24 \ $
  or
  $\ 240 \,$,
  resp., and
  \begin{equation*}
    E(\Ker(\underline{h}')) \, = \, 2 \cdot \pi_m(S^n) \, \neq \, \pi_m(S^n) \, = \, E(\pi_{m-1}(S^{n-1})) \,.
  \end{equation*}
  If
  $\ \#G \, \leq \, 2 \ $
  then
  $\ \MC(f_1, f_2) \, < \, \infty \ $
  for all maps
  $\ f_1, f_2 \, \colon \, S^m \, \longrightarrow \, S^n \diagup G \,$.
  However, if
  $\ \#G \, \geq \, 3 \ $
  and
  $\ [\widetilde{f}_1] - [\widetilde{f}_2] \, \notin \, 2 \pi_m(S^n) \ $
  then
  $\ \MC(f_1, f_2) \, = \, \infty \,$.
  \qed
\end{exa}

\vspace{1ex} Finally let us take a look at maps into surfaces.
\begin{exa}\label{exa:1.34}
  {\bfseries $\boldsymbol{m \, > \, n \, = \, 2 }$.}\\
  Given maps
  $\ f_1, f_2 \ $
  fom
  $\ S^m, \, m \, > \, 2 \,$,
  into any closed surface
  $\ Y \,$,
  we have
  \begin{align*}
    \MC(f_1, f_2) \, &= \, \begin{cases}
                             \infty &   \text{if } f_1 \, \not\sim \, f_2 \,; \\
                             0      &   \text{if } f_1 \, \sim \, f_2 \,;
                           \end{cases}\\
    \intertext{and}
    \MCC(f_1, f_2) \, &= \, N_0(f_1, f_2) \, = \, \begin{cases}
                                                   \# \pi_1(Y)  &   \text{if } f_1 \, \not\sim \, f_2 \,; \\
                                                   0        &   \text{if } f_1 \, \sim \, f_2 \,.
                                                 \end{cases}
  \end{align*}

  The same result holds for maps from
  $\ S^m\ $
  into an $n$--dimensional spherical space form whenever
  $\ m, n \, \geq \, 2 \ $
  and
  $\ \pi_{m-1}(S^{n-1}) \, = \, 0 \,$,
  e.g. when
  $\ [f_1], [f_2] \, \in \, \pi_{22}(S^{10} \diagup G) \, \cong \, \Z_{12} \,$.
  \qed
\end{exa}

\vspace{1ex} When the target manifold of our maps is not a spherical
space form certain finiteness conditions for the minimum number $\
\MC\ $ can still be expressed in terms of Hopf--Ganea invariants
(cf. \cite{K3}, corollary 7.4 and theorem 7.6).

When the target manifold is a (real, complex or quaternionic)
projective space a detailed discussion of minimum numbers and
certain Nielsen numbers was carried out e.g in \cite{K5} and
\cite{K7}.

\begin{problem}\label{problem:1.35}
  Let
  $\ f_1, f_2 \, \colon \, X \, \longrightarrow \, Y \ $
  be maps between arbitrary smooth connected manifolds,
  $\ X\ $
  being compact.

  Is
  $\ \MC(f_1, f_2) \, = \, \MCC(f_1, f_2) \ $
  whenever
  $\ \MC(f_1, f_2) \, < \, \infty \,$?
  Also: give general complete criteria for
  $\ \MC(f_1, f_2) \ $
  being finite.

  Recall that the case
  $\ n \, = \, 2\ $
  may play a special role here (cf. \cite{K3}, theorem 1.2(iii)).
\end{problem}

\begin{con_and_not}
  Throughout this paper
  $\ m, n \, \geq \, 1 \,$,
  and
  $\ Y\ $
  is a smooth connected n--dimensional manifold (Hausdorff, having a countable base) without boundary and with basepoint
  $\ y_0 \, $.
  Let
  $\ \Omega Y \, = \, \Omega(Y,y_0) \ $
  (and
  $\ (\Omega Y)^+ \,$,
  resp.) denote the loop space (with a single disjoint point added, resp.).
  $\ E\ $
  stands for the Freudenthal suspension.
  $\ T_y(Y)\ $
  is the tangent space of
  $\ Y\ $
  at a point
  $\ y \, \in \, Y \, $.
  If
  $\ p \, \colon \, \widetilde{Y} \, \longrightarrow \, Y \ $
  denotes the universal covering we equip
  $\ \widetilde{Y}\ $
  with a basepoint
  $\ \widetilde{y}_0 \, \in \, p^{-1} \lt \lbrace y_0 \rbrace \rt $.
  Identity maps are denoted by
  $\ \id \,$.
  The symbols
  $\ \sim \ $
  (or
  $\ \not\sim \,$,
  resp.) mean freely homotopic (or not, resp.).
  $\ \# S\ $
  is the cardinality of a set
  $\ S\,$.
\end{con_and_not}

\begin{acknowledgement}
  It is a pleasure to thank Marek Golasinski for very valuable references.
\end{acknowledgement}
\section{The group $\ \boldsymbol{\pi_m(S^q \wedge (\Omega Y)^+)}\ $ and the partial suspension homomorphism \ e.}\label{sec:2}

\indent

In our discussion of Nielsen numbers a central r\^ole will be played by Thom spaces of the form 
$\ S^q \wedge ((\Omega Y)^+) \ $.
In this section we interpret these as fibers of appropriate fibrations.
This will allow us to study their homotopy groups, as well as suspension homomorphisms which are important in coincidence theory.

Fix integers 
$\ m,\ q\ \geq\ 1\ $
and base points
$\ \infty \ \in\ S^q,\ y_0 \in Y $.
Then the obvious collapsing map
\begin{equation}\label{equ:2.1}
  p_2 \ :\ S^q \vee \ Y \ \longrightarrow \ Y
\end{equation}
can be transformed (up to homotopy equivalences) into the fibration
$\ \mathrm{ev}_1\ $
in the top line of the homotopy commutative diagram
\begin{equation}
  \begin{aligned}\label{equ:2.2}
    \xymatrix @R=1.0cm @C=2.1cm {
      \ \ F \ \                                 \ar@<-5.5pt>@{}[r]^-{\subset}            \ar[d]^-{\text{quot}}_{\rotatebox{270}{$\sim$}}  &
      \ \ Z \ \                                 \ar[r]^-{\mathrm{ev}_1}         \ar[d]_{\rotatebox{270}{$\sim$}}  &
      \ \ Y \ \                                 \ar@{=}[d]      \\
      \ \ S^q \wedge ((\Omega Y)^+) \ \         \ar[r]_-{j}     &
      \ \ S^q \vee Y \ \                        \ar@<-2pt>[r]_-{p_2}   &
      \ \ Y \ \                                 \ar@{-->}@<-2pt>[l]_-{incl_Y}
    }
  \end{aligned}
\end{equation}
Here
\begin{equation}\label{equ:2.3}
  Z\ :=\ \left\lbrace (x, \theta) \;\in\; (S^q \vee Y) \,\times\, Y^I \ \left|\right.\  \theta(0) \;=\; p_2(x) \ \right\rbrace
\end{equation}
is homotopy equivalent to
$\ S^q \vee Y\ $
(via the first projection); the fiber map
$\ \mathrm{ev}_1\ $
evaluates the path
$\ \theta \ \text{ at } 1\ \in\ I := [0,\,1]$.
The fiber
\begin{equation}\label{equ:2.4}
  F\ = \ \left\lbrace (x, \theta) \;\in\; (S^q \vee Y) \,\times\, Y^I \ \left|\right.\  \theta(0) \;=\; p_2(x),\, \theta(1)=y_0 \  \right\rbrace
\end{equation}
contains the contractible subspace
$\ P = \ \left\lbrace \, (x, \theta)\, \in \, F \ \vert \ x \,\in\, Y \, \right\rbrace $,
and the quotient map
\begin{equation}\label{equ:2.5}
  \text{quot} \ \colon \ F \longrightarrow F / P \; = \; S^q \wedge ((\Omega Y)^+)
\end{equation}
is a homotopy equivalence (compare \cite{C}, p.2769,
and \cite{K3}, lemma 7.1).
Its inverse (when composed with the fiber inclusion) yields the map
\begin{equation}\label{equ:2.6}
  j \ \colon \ S^q \wedge ((\Omega Y)^+) \longrightarrow S^q \vee Y \ .
\end{equation}
Obviously the collapsing map
$\ p_2\ $
(cf. (\ref{equ:2.1})) allows a canonical right inverse.
Thus the exact homotopy sequence of the fibration
$\ \mathrm{ev}_1 \ $
splits and takes the following form:
\begin{equation}
  \begin{aligned}\label{equ:2.7}
    \xymatrix @R=0.5cm @C=1.0cm {
    0   \ar[r]  &
    \pi_m(S^q \wedge ((\Omega Y)^+))    \ar[r]^-{j_*}   &
    \pi_m(S^q \vee Y)   \ar@<2pt>[r]^-{p_{2*}}        &
    \pi_m(Y)    \ar@{-->}@<2pt>[l]^-{\text{incl}_Y*} \ar[r]  &
    0
    }
  \end{aligned} \ .
\end{equation}
We conclude

\begin{prop}\label{prop:2.8}
  The map
  $\, j\, $
    (cf. (\ref{equ:2.6})) induces the isomorphism
  \begin{equation*}
    \xymatrix @R=0.1cm @C=1.0cm {
      j_* \, \colon \, \pi_m\left( S^q \wedge (( \Omega Y)^+ ) \right) 	\ar[r]^-{\cong}	 &
      \Ker_{m,q}(Y) \ :=\ \Ker(\ p_{2*}\ \colon\ \pi_m( S^q \vee Y ) \longrightarrow \pi_m(Y) \ ).
    }
  \end{equation*}
\end{prop}

\vspace{2ex}
It will be useful to describe
$\ j_*\ $
geometrically.
Let
$\ B^q(r) \ $
(and
$\ \partial B^q(r), \ $
resp.),
$\ r\, > \,0,\ $
denote the compact ball (and sphere, resp.) of radius
$\ r\ $
in
$\ \R^q,\ $
and use an (orientation preserving) standard identification
$\ B^q(r) \diagup \partial B^q(r) \ =\ S^q.\ $
Given a base point preserving map
\begin{equation*}
  u\ \colon\ S^m \ =\ \R^m \cup \lbrace\infty\rbrace \ \longrightarrow\ 
  \left( B^q(1) \times \Omega Y \right) \diagup
  \left( \partial B^q(1) \times \Omega Y \right) \ =\ S^q \wedge \left( (\Omega Y)^+ \right) ,
\end{equation*}
we may deform it until we have the following standard situation (as in Pontrjagin--Thom theory):
there is a smoothly embedded tubular neighbourhood
$\ T_3\ :=\ B^q(3) \, \times C \; \subset \R^m \ $
of
$\ C\ :=\ u^{-1} \left( \lbrace0\rbrace \times \Omega Y \right)\ $
such that
\begin{equation}\label{equ:2.9}
  u(x)\ =\ 
  \begin{cases}
    \left[ ( v, \theta_{u(c)} ) \right]  &  \text{if } x = (v, c) \in B^q(1) \times C; \\
    \infty                               &  \text{if } x \notin \mathring{B}^q(1) \times C.
  \end{cases}
\end{equation}
Here
$\ \infty\ $
denotes also the base point of the Thom space
$\ \left( \R^q \times \Omega Y \right) \cup \lbrace\infty\rbrace \ =\ S^q \wedge \left( (\Omega Y )^+ \right),\ $
and
$\ \theta_{u(c)} \ =\ u(0,c)\; \in \; \Omega Y \ (= \lbrace 0 \rbrace \times \Omega Y ).\ $
Thus
$\ u\ $
maps all normal slices
$\ \mathring{B}^q(1) \times \lbrace c \rbrace \text{ in } T_3 \ $
by the same diffeomorphism to the corresponding fibers
$\ \R^q \times \lbrace\theta_{u(c)} \rbrace \ $
in the Thom space,
$\ c \in C. \ $

Next consider the map
\begin{equation*}
  u'\ \colon \  S^m\ \longrightarrow \ S^q \vee Y
\end{equation*}
defined by
\begin{equation}\label{equ:2.10}
  u'(x) \ =\ 
  \begin{cases}
    [v] \in B^q(1) \diagup \partial B^q(1) = S^q        & \quad \text{if }\, x \ =\ (v,c) \in B^q(1) \times C; \\
    \theta_{u(c)}(\parallel v \parallel - 1) \in Y      & \quad \text{if }\, x \ =\ (v,c) \in \left(B^q(2) \setminus \mathring{B}^q(1)\right) \times C; \\
    \text{wedge point of } S^q \vee Y                   & \quad \text{if }\, x \ \notin B^q(2) \times C.
  \end{cases}
\end{equation}

\addtocounter{figure}{10}
\begin{figure}[htb]
  \centering
  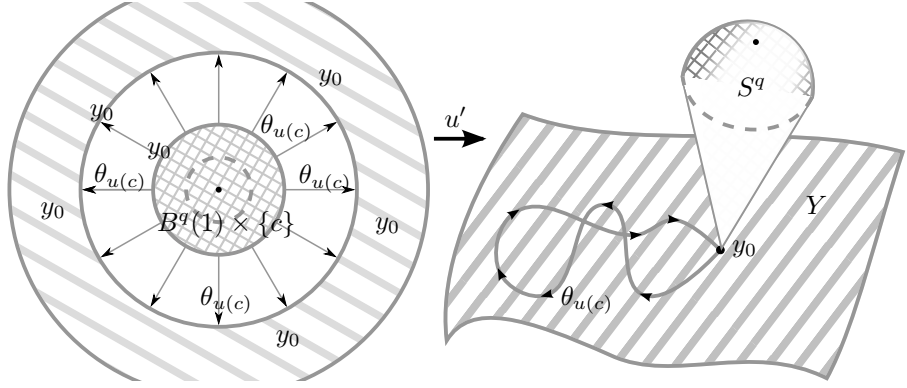
  \caption{The image of $\ u'\ $ on any normal slice $\ B^q(3) \times \lbrace c \rbrace,\; c \in C$.}
  \label{fig:figure2.11}
\end{figure}

\stepcounter{equation}
\begin{prop}\label{prop:2.12}
  $\ j_* ([u])\ =\ [u'].\ $
\end{prop}

\begin{proof}
  We need only to lift
  $\ u' \ $
  to a map
  \begin{equation*}
    \widetilde{u}'\ \colon\ S^m \longrightarrow F \; \subset \; Z
  \end{equation*}
  such that
  $\ \quot \comp \widetilde{u}' \; \sim \; u \ $
  (compare diagram \ref{equ:2.2}).
  In view of the standard form of
  $\ u\ $
  and
  $\ u'\ $
  we can do so slice by slice.
  Given
  $\ c \in C,\ $
  let
  $\ x = (v,c)\ $
  lie in the normal slice
  $\ B^q(3) \times \lbrace c \rbrace\ $
  in the tubular neighbourhood
  $\ T_3.\ $
  Then we must find a path
  $\ \theta\ $
  in
  $\ Y\ $
  starting from
  $\ p_2\left(u'(x)\right)\ $
  and ending at
  $\ y_0.\ $
  If
  $\ \parallel v \parallel \leq 1, \ $
  then
  $\ p_2\left(u'(x)\right) \ =\ y_0\ $
  and we put
  $\ \theta \ = \ \theta_{u(c)}\ $
  (compare (\ref{equ:2.9}) and (\ref{equ:2.10})).
  If
  $\ 1 \leq \parallel v \parallel \leq 2, \ $
  then
  $\ p_2\left(u'(x)\right) \ =\ \theta_{u(c)}( \parallel v \parallel - 1),\ $
  and we define 
  $\ \theta\ $
  to be the path which first goes {\itshape back} to
  $\ y_0\ $
  along
  $\ \theta_{u(c)}\ $
  and then traverses the full loop
  $\ \theta_{u(c)}.\ $
  In particular, if
  $\ \parallel v \parallel \ =\ 2, \ $
  then
  $\ \theta\ $
  is
  $\ \theta^{-1}_{u(c)}\ $
  followed by
  $\ \theta_{u(c)}. \ $
  We use the remaining parameter
  $\ 2\ \leq\ \parallel v \parallel \ \leq\ 3\ $
  in the outer part
  $\ \left( B^q(3) \setminus B^q(2) \right) \times \lbrace c \rbrace \ $
  of our normal slice to deform
  $\ \theta^{-1}_{u(c)} \cdot \theta_{u(c)} \ $
  in
  $\ \Omega(Y)\ $
  to the constant loop.
  
  This procedure allows us to construct a continuous lifting of
  $\ u'\ $
  on the whole tubular neighbourhood
  $\ T_3\ =\ B^q(3) \times C,\ $
  and it can be extended trivially to all of
  $\ S^m. \ $
  All but the innermost part
  $\ B^q(1) \times C\ $
  of
  $\ T_3\ $
  gets mapped to
  $\ P\ $
  (cf. (\ref{equ:2.5})) so that
  $\ \quot \comp \widetilde{u}'\ =\ u\ $
  and
  $\ j_*([u])\ =\ (j \comp \quot)_*([\widetilde{u}'])\ =\ [u'],\ $
  as required.
\end{proof}

Next we construct a partial suspension homomorphism\footnote{After I had written this paper M. Golasinski drew my attention to the work of H. J. Baues who had introduced partial suspensions for suitable spaces $ A, B $ and $ Y $ (cf. \cite{Ba}, chapter 3).
My explicit geometric construction turns out to agree with Baues' homotopy theoretical definition for the case $ A = S^m, B = S^q $.}
\begin{equation}
  \begin{aligned}\label{equ:2.13}
    \xymatrix @R=0.5cm @C=1.0cm {
    e\ \colon \ \Ker_{m,q}(Y)   \ar@<-5.5pt>@{}[d]^-{\rotatebox{270}{$\subset$}}   \ar[r] &
    \Ker_{m+1,q+1}(Y)           \ar@<-5.5pt>@{}[d]^-{\rotatebox{270}{$\subset$}}   \\
    \pi_m(S^q \vee Y)           &
    \pi_{m+1}(S^{q+1} \vee Y)
    }
  \end{aligned}
\end{equation}
(compare proposition \ref{prop:2.8}) which suspends
$\ S^q,\ $
but leaves
$\ Y\ $
unchanged.
We use the same approach and notations as in (\ref{equ:2.9}) and (\ref{equ:2.10}).

Given
$\ [w] \,\in\, \Ker_{m,q}(Y),\ $
we may assume that
\begin{equation}\label{equ:2.14}
  w\ \colon\ S^m \;=\; \R^m \cup \lbrace\infty\rbrace \ \longrightarrow\ \left( B^q(1) \diagup \partial B^q(1) \right) \vee Y \ =\ S^q \vee Y
\end{equation}
has the following standard form:
there is a tubular neighbourhood
$\ T_1\ =\ B^q(1) \times C \ \subset \R^m \ $
of
$\ C\ =\ \lbrace 0 \rbrace \times C \ $
such that
\begin{enumerate}[label=(\roman*)]
  \item $\ w \vert T_1\ $ is the obvious composed projection from
    $\ B^q(1) \times C \ $
    to
    $\ B^q(1) \diagup \partial B^q(1)\ =\ S^q \subset S^q \vee Y; \ $
    and
  \item $\ w \left( S^m \setminus \left( \mathring{B}^q(1) \times C \right) \right) \ \subset\  Y \ \subset\  S^q \vee Y. $
\end{enumerate}
Pick a base point preserving nullhomotopy
$\ W\ \colon\ S^m \times [0,1] \longrightarrow Y \ $
from
$\ p_2 \comp w \ $
to the constant map and define
\begin{equation*}
  e_W \vert \ \; \colon \ \; \R^m \times [-2,2] \ \; \longrightarrow \ \; S^{q+1} \vee Y
\end{equation*}
\begin{enumerate}[label=(\roman*)]
  \item on
    $\ T_1 \times [-1,1] \;=\; \left( B^q(1) \times [-1,1] \right) \times C \ $
    by the obvious projection to
    \begin{equation*}
      \left( B^q(1) \times [-1,1] \right) \diagup \partial \left( B^q(1) \times [-1,1] \right) \;=\; S^{q+1} \; \subset \; S^{q+1} \vee Y;
    \end{equation*}
  \item on
    $\ \left( \R^m - T_1 \right) \times [-1,1] \ $
    by the projection to
    $\ \R^m - T_1,\ $
    composed with the restricted map
    $\ w\vert \ $
    into
    $\ Y \subset S^{q+1} \vee Y;$
  \item for
    $\ \left( x, x_{m+1} \right) \,\in\, \R^m \times \R \ $
    with
    $\ 1 \leq \left| x_{m+1} \right| \leq 2 \ $
    by
    \begin{equation*}
      e_W(x, x_{m+1}) \,=\, W(x, \left| x_{m+1} \right| -1 ) \ \  \in \ \  Y \  \subset\  S^{q+1} \vee Y \, .
    \end{equation*}
\end{enumerate}
These piecewise definitions fit well together and allow a trivial extension
$\ e_W \ $
to all of
$\ S^{m+1} \,=\, (\R^m \times \R) \cup \lbrace\infty\rbrace.$

The resulting homotopy class
\renewcommand{\theequation}{\arabic{section}.\arabic{equation}'}
\addtocounter{equation}{-2}
\begin{equation}\label{equ:2.13Strich}
  e\left( [w] \right) \ :=\ [e_W] \; \in \; \pi_{m+1} \left( S^{q+1} \vee Y \right)
\end{equation}\numberwithin{equation}{section}\addtocounter{equation}{1}
is independent of our choice of
$\ W.\ $
Indeed, given another nullhomotopy
$\ W' \ $
of
$\ p_2 \comp w, \ $
let
$\ [ W^{-1} \cdot W' ] \; \in \; \pi_{m+1}(Y) \; \subset \; \pi_{m+1}(S^{q+1} \vee Y) \ $
be defined by concatenation; then
\begin{equation*}
  [e_{W'}] \ =\ -[ W^{-1} \cdot W' ] + [e_W] + [ W^{-1} \cdot W' ] \ =\ [e_{W}].
\end{equation*}
Similarly,
$\ p_{2*}([e_W]) \ =\ 0, \ $
again due to the symmetry property of our construction with respect to the variable
$\ x_{m+1}.$

Thus we obtain a welldefined partial suspension homomorphism
$\ e\ $
as in (\ref{equ:2.13}).
Clearly
$\ e\ $
restricts to the full (standard) suspension on the subgroup
$\ \pi_m(S^q)\ $
of
$\ \pi_m(S^q \vee Y).$\\

\begin{rem}\label{rem:2.15}
  The representation of
  $\ e([w])\ $
  need not to be quite so specific as in (\ref{equ:2.13Strich}).
  Let
  \begin{equation*}
    \widehat{T}_1 \ =\ B^{q+1}(1) \times C \ \  \hookrightarrow \ \  \R^m \times \R \ =\ \R^{m+1} \ \subset \ S^{m+1}
  \end{equation*}
  be a tubular neighbourhood inclusion with extends the inclusion of
  $\ T_1 \ =\ B^q(1) \times C\ $
  into
  $\ \R^m \; = \; \R^m \times \lbrace 0 \rbrace \ $
  and takes the last coordinate in
  $\ B^{q+1}\ $
  to
  $\ x_{m+1}. \ $
  Also, let
  \begin{equation*}
    \widehat{W} \ \ \colon \ \ S^{m+1} \setminus \widehat{T}_1 \ \ \longrightarrow \ \ Y \ \subset \ S^{q+1} \vee Y
  \end{equation*}
  be any map which extends
  $\ w\vert S^m \setminus T,\ $
  maps the boundary
  $\ \partial \widehat{T}_1\ $
  of
  $\ \widehat{T}_1\ $
  to the wedgepoint and satisfies the {\itshape symmetry condition}
  \begin{equation*}
    \widehat{W}(x, x_{m+1} ) \ =\ \widehat{W}(x, -x_{m+1} )
  \end{equation*}
  for all
  $\ (x, x_{m+1} ) \,\in\, \R^m \times \R \subset S^{m+1}, \ (x, x_{m+1}) \notin \widehat{T}_1. \ $
  Using a suitable ambient deformation of
  $\ \widehat{T}_1 \ $
  into
  $\ T_1 \times [-1,1] \ $
  it is not hard to see that the map
  \begin{equation*}
    e_{\widehat{W}} \ \colon\ S^{m+1} \ \longrightarrow\ S^{q+1} \vee Y,
  \end{equation*}
  defined by
  $\ \widehat{W} \ $
  and the projection
  \begin{equation*}
    \widehat{T}_1 \ =\ B^{q+1} \times C \ \longrightarrow\ B^{q+1}(1) \diagup \partial B^{q+1}(1) \ = \ S^{q+1} \subset S^{q+1} \vee Y,
  \end{equation*}
  represents
  $\ e([w]).$
  \qed
\end{rem}

\begin{thm}\label{thm:2.16}
  The partial suspension homomorphism
  $\ e\ $
  has the following properties (where
  $\ E\ $
  denotes (standard) full suspension homomorphisms and
  $\ m, m', q, q' \geq 1$):
  \begin{enumerate}[label=(\alph*)]
    \item \label{thm:2.16a} {\bfseries Compatibility with the isomorphism}
      $\boldsymbol{\ j_*.\ }$
      The diagram
      \begin{equation*}
        \xymatrix @R=1.0cm @C=2.1cm {
          \pi_m(S^q \wedge ( (\Omega Y)^+ ) )    \ar[r]^-{\cong}_-{j_*}  \ar[d]_-{E} &
          \Ker_{m,q}(Y)         \ar[d]_-{e} \\
          \pi_{m+1}(S^{q+1} \wedge ( (\Omega Y)^+ ) )     \ar[r]^-{\cong}_-{j_*} &
          \Ker_{m+1,q+1}(Y)
        }
      \end{equation*}
      (cf. Proposition \ref{prop:2.8} and (\ref{equ:2.13}) ) commutes.
    \item \label{thm:2.16b} {\bfseries Naturality.}
      \begin{enumerate}[label=(\roman*)]
        \item Given a base point preserving map
          $\ g\ \colon\ S^{m'} \longrightarrow S^m\ $
          and
          $\ [w] \: \in \: \Ker_{m,q}(Y)\,$,
          we have
          \begin{equation*}
            e( [ w \comp g ] ) \ = \ (e ([w]) ) \comp [E \, g] \: \in \: \Ker_{m'+1,q+1}(Y) \, .
          \end{equation*}
        \item Given base point preserving maps
          $\ g_1 \ \colon \ S^q \longrightarrow S^{q'} \ $
          and
          $\ g_2 \ \colon \ Y \longrightarrow Y' \ $
          between manifolds, the map
          $\ g_1 \vee g_2 \ \colon \ S^q \vee Y \longrightarrow S^{q'} \vee Y' \ $
          induces the commuting diagram
          \begin{equation*}
            \xymatrix @R=1.0cm @C=1.5cm {
              \Ker_{m,q}(Y)     \ar[r]^-{e} \ar[d]^-{ (g_1 \vee g_2)_* } &
              \Ker_{m+1,q+1}(Y)     \ar@<-5.5pt>@{}[r]^-{\subset} \ar[d]^-{ \left( \left( E g_1 \right) \vee g_2 \right)_* } &
              \pi_{m+1}(S^{q+1} \vee Y) \\
              \Ker_{m,q'}(Y')     \ar[r]^-{e} &
              \Ker_{m+1,q'+1}(Y')     \ar@<-5.5pt>@{}[r]^-{\subset} &
              \ \ \pi_{m+1}(S^{q'+1} \vee Y').
            }
          \end{equation*}
      \end{enumerate}
    \item \label{thm:2.16c} {\bfseries Compatibility with Whitehead products.}
    
      Given
      $\ \alpha \, \in \, \pi_m(S^q \vee Y) \,=\, \pi_m(Y) \oplus \Ker_{m,q}(Y) \ $
      and
      $\ \beta \, \in \, \Ker_{m',q}(Y), \ $
      we have:
      \begin{enumerate}[label=(\roman*)]
        \item if
          $\ \alpha \, \in \, \pi_m(Y), \ $
          then
          $\ \ e( [\alpha, \beta] ) \, = \, \pm [\alpha, e(\beta) ] \ \ \in \  \Ker_{m+m',q+1}(Y); \ $
        \item if
          $\ \alpha \, \in \, \Ker_{m,q}(Y), \ $
          then
          $\ \ e( [\alpha, \beta] ) \, = \, 0. \ $
      \end{enumerate}
  \end{enumerate}
\end{thm}

\providecommand{\KoschKursiv}[1]{\ensuremath{\mathit{#1}}}
\begin{proof}
  Given
  $\ [u] \,\in\, \pi_m(S^q \wedge (\Omega Y)^+ ),\ $
  pick a representative
  $\ u\ $
  in standard form (as in (\ref{equ:2.9}) ), based on a map
  $\ g \,\colon\, C \longrightarrow \Omega Y\ $
  and on a framed embedding
  $\ C \subset \R^m. \ $
  Then we can represent the suspension
  $\ E( [u] )\ $
  by a map
  $\ E u\ $
  in standard form, based on the same
  $\ g\ $
  and on the composite embedding
  $\ C \subset \R^m \subset \R^{m+1}. \ $
  Now compare the corresponding maps
  $\ u'\ $
  and
  $\ (E u)'\ $
  in standard form (cf. (\ref{equ:2.10}) and (\ref{equ:2.14}) ) and apply remark \ref{rem:2.15} to
  $\ w \,:=\, u'. \ $
  When we restrict
  $\ (E u)'\ $
  to the complement of the tubular neighbourhood
  $\ \widehat{T}_1\ $
  of
  $\ C\ $
  in
  $\ S^{m+1}, \ $
  we obtain a map
  $\ \widehat{W}\ $
  as in remark \ref{rem:2.15}.
  Thus
  \begin{equation*}
    (e \cdot j_*([u]) \,= ) \quad  e([u']) \,=\, [ (E u)' ] \quad (=\, j_* \comp E([u]) ).
  \end{equation*}
  This establishes our first claim.
  
  Naturality follows similarly from the way
  $\ e\ $
  is defined or from remark \ref{rem:2.15}.
  
  For the proof of our third claim we use the geometric description of Whitehead products suggested e.g. by \cite{W}, Ch. X, (7.1) or Figure 10.2.
  Write
  $\qquad \R^{m+m'-1} \ =\  \R^{m-1} \ \times \ \R \ \times \ \R^{m'-1} \ $
  and let
  $\ S' \subset \R^{m+m'-1} \ $
  denote the unit sphere (with center
  $\ 0\ $)
  of
  $\ \R^{m-1} \times \R \times \lbrace 0 \rbrace, \ $
  framed in the standard fashion by the outward pointing vector and
  $\ \R^{m'-1}. \ $
  Similarly, let
  $\ S \subset \lbrace 0 \rbrace \times \R \times \R^{m'-1} \ $
  be the framed unit sphere (with center
  $\ e_m = (0,1,0) \in S'\ $
  ) of the normal space of
  $\ S'\ $
  at
  $\ e_m. \ $
  Also let
  \begin{equation*}
    T \ :=\ B^m \times S,\quad T' = B^{m'} \times S' \quad \subset \ \R^{m+m'-1}
  \end{equation*}
  denote compact tubular neighbourhoods of
  $\ S\ $
  and
  $\ S',\ $
  resp., parametrized compatibly with the framings and disjoint (but linked).
  
  Now pick representatives \KoschKursiv{a}, \KoschKursiv{b} in standard form (cf. (\ref{equ:2.4}) ) of the homotopy classes
  $\ \alpha, \beta \,\in\, \pi_*(S^q \vee Y).\ $
  Define
  \begin{equation*}
    w_{\KoschKursiv{a}, \KoschKursiv{b}} \, \colon\, S^{m+m'-1} \longrightarrow S^q \vee Y
  \end{equation*}
  on the tubular neighbourhood
  $\ T\ $
  and
  $\ T',\ $
  resp., by composing \KoschKursiv{a} and \KoschKursiv{b}, resp., with the obvious projections (e.g. compose
  $\ \alpha \ $
  with
  \begin{equation*}
    T \,=\, B^m \times S \longrightarrow B^m \diagup \partial B^m \,=\, S^m \text{ )} \, ,
  \end{equation*}
  and on
  $\ S^{m+m'-1} - (T \cup T')\ $
  by the constant map.
  Then
  \begin{equation*}
    [ w_{\KoschKursiv{a}, \KoschKursiv{b}} ] = \pm [ \alpha, \beta ]
  \end{equation*}
  and
  $\ w_{\KoschKursiv{a}, \KoschKursiv{b}}\ $
  is again in standard form.
  
  Now construct
  $\ e([\KoschKursiv{b}]) \,=\, [e_{\widehat{W}} ] \ $
  as in remark \ref{rem:2.15} and consider the map
  \begin{equation*}
    \xymatrix @R=1.0cm @C=1.0cm {
      \widehat{T}' = B^{m'+1} \times S' \ar[r]  &
      S^{m'+1}  \ar[r]^-{e_{\widehat{W}}}       &
      S^q \vee Y
    }
  \end{equation*}
  which extends
  $\ w_{\KoschKursiv{a}, \KoschKursiv{b}} \vert T' \ $
  to a tubular neighbourhood of
  $\ S'\ $
  in
  $\ \R^{m+m'}. \ $
  If \KoschKursiv{\ a\ } maps
  $\ S^m\ $
  fully into
  $\ Y, \ $
  we can also extend
  $\ w_{\KoschKursiv{a}, \KoschKursiv{b}} \vert T \ $
  to a tubular neighbourhood
  $\ \widehat{T} \,=\, B^m \times \widehat{S} \ $
  of the unit sphere
  $\ \widehat{S} \ $
  (around
  $\ e_m\ $)
  in
  $\ \left( \lbrace 0 \rbrace \times \R \times \R^{m'-1} \right) \times \R \ \subset \ \R^{m+m'} \ $
  by applying \KoschKursiv{\ a\ } to each normal slice.
  We get a representative of
  $\ e([\alpha, \beta]) \ $
  (as in remark \ref{rem:2.15}) which represents also
  $\ \pm [ \alpha, e(\beta) ].\ $
  This proves the first part of claim \ref{thm:2.16c}.
  
  If
  $\ \alpha \,\in\, \Ker_{m,q}(Y) \ $
  we do not need all of
  $\ \widehat{S}\ $
  but we can extend both
  $\ w_{a, b} \vert T \ $
  and
  $\ w_{a, b} \vert T' \ $
  to tubular neighbourhoods of
  $\ S, S' \subset \R^{m+m'-1} \ $
  in
  $\ \R^{m+m'}. \ $
  But these tubular neighbourhoods are not linked and can be isotoped to disjoint
  $\ x_{m+m'}$--levels.
  Thus
  \begin{equation*}
    e([\alpha, \beta]) \,=\, [0, e(\beta) ] + [e(\alpha), 0] \,=\, 0.   \qedhere
  \end{equation*}
\end{proof}

\begin{cor}\label{cor:2.17}
  Let
  $\ [w] \,\in\, \pi_*(S^q \vee Y) \ $
  be an iterated Whitehead product with factors in
  $\ \pi_*(Y)\ $
  and with at least one factor purely in
  $\ \pi_*(S^q). \ $
  If
  $\ [w]\ $
  has precisely one factor
  $\ [v] \,\in\, \pi_*(S^q), \ $
  then
  $\ \pm e([w])\ $
  equals the same Whitehead product, but with
  $\ [v]\ $
  replaced by the standard suspension
  $\ E([v]) \,\in\, \pi_*(S^{q+1} ); \ $
  otherwise
  $\ e([w]) \,=\, 0 $.
\end{cor}

\begin{proof}
  This follows by applying theorem \ref{thm:2.16}\ref{thm:2.16c} and the anticommutativity of Whitehead products repeatedly.
\end{proof}

\section{The coincidence invariants $\ \boldsymbol{\omega_r }\ $ and Hopf--Ganea homomorphisms.}\label{sec:3}

\indent

In this section we discuss our $\omega$--invariants and their interpretation -- via the isomorphism
$\ j_* \ $
(cf. proposition \ref{prop:2.8}) -- in terms of suspensions and Hopf--Ganea invariants.

Assume that
$\ n \geq 2 \, $.
Fix a local orientation of the $n$-manifold
$\ Y \ $
at its basepoint
$\ y_0 \ $
and an embedded path
$\ \gamma \ $
in
$\ Y \ $
from
$\ y_0 \ $
to some point
$\ * \; \in \; Y \, ,\ * \neq y_0 \,$.
(Constant maps with value
$\ * \ $
will also be denoted by
$\ * \,$).
Then, given
$\ [f_1] \; \in \; \pi_m(Y, *) \ $
and
$\ [f_2] \; \in \; \pi_m(Y, y_0) \, $,
we can use the Pontryagin--Thom procedure to interpret
$\ \omega^{\#}(f_1, f_2) \; = \; \omega_0(f_1, f_2) \ $
(cf. \ref{equ:1.6}) as an element in
$\ \pi_m(S^n \wedge (\Omega Y)^+) \ $
(cf. \cite{K3},
proposition 2.5); more generally,
\begin{equation}\label{equ:3.1}
  \omega_r(f_1, f_2) \; = \; E^r(\omega^{\#}(f_1, f_2)) \; \in \; \pi_{m+r}(S^{n+r} \wedge (\Omega Y)^+) \, , \ r = 0, 1, \ldots, \infty \, ,
\end{equation}
(compare \ref{equ:1.6} and \ref{equ:3.1Strich} below) where
$\ E \ $
denotes the suspension homomorphism.
In addition
\begin{equation}\label{equ:3.2}
  \omega_r(f_1 + f_1', f_2 + f_2') \; = \; \omega_r(f_1, f_2) + \omega_r(f_1', f_2')
\end{equation}
for all
$\ [f_1], [f_1'] \; \in \; \pi_m(Y, *) \, , \ [f_2], [f_2'] \; \in \; \pi_m(Y, y_0) \ $
(cf. \cite{K3}, 6.1).

Let us describe these invariants more explicitly in the case when
$\ f_1 \; \equiv \; * \, $.
Given
$\ [f] \; \in \; \pi_m(Y, y_0), \ $
we may assume that
$\ f \; \colon \; S^m \;\longrightarrow\; Y \ $
is smooth with regular value
$\ * $.
Then
$\ \omega^{\#}(*, f)\ $
is the {\itshape nonstabilized} bordism class of the triple
$\ (\, C, \, \widetilde{g},\, \bar{g}^{\#} \,) \ $
consisting of
\begin{enumerate}[label=(\roman*)]
  \item the embedded smooth submanifold
    $\ C \; := \; f^{-1}( \lbrace * \rbrace ) \ $
    of
    $\ \R^m \ \subset \ \R^m \cup \ \lbrace \infty \rbrace \ =\ S^m \,$;
  \item the map
    $\ \widetilde{g}\ $
    from
    $\ C \ $
    to the loop space
    $\ \Omega Y \ $
    defined as follows: pick a 
    homotopy
    $\ G \; \colon \; C \times I \longrightarrow S^m \ $
    from the inclusion
    $\ C \; \subset \; S^m \ $
    to the constant map with value
    $\ \infty \ $
    (= the basepoint of
    $ \; S^m \,$).
    Then
    $\ \widetilde{g}(x) \ $
    is the (concatenated) loop
    $\ \gamma \, \cdot \, f\left(G \,\left(x,-\right)\right) ,\; x \; \in \; C $;
  \item
    $ (-1) \cdot \bar{g}^{\#}\ $
    is the isomorphism from the normal bundle of
    $\ C \ $
    in
    $\ \R^m \ $
    to the trivial bundle (over
    $\ C \;$)
    with fiber
    $\ T_*(Y) \ \cong \ T_{y_0}(Y) \ \cong\ \R^n ,\ $
    induced by the tangent map of
    $\ f \ $
    and the chosen path
    $\ \gamma \,$.
\end{enumerate}

We use the Pontryagin--Thom
procedure to identify the group of nonstabilized bordism classes of such triples with the homotopy group
$\ \pi_m(S^n \wedge \, (\,( \Omega Y )^+ ) \,) \ $
(for more details see \cite{K3}).

If we forget about embeddings and consider
$\ C \ $
only as an {\itshape abstract}
$(m-n)$--dimensional manifold, equipped with the map
$\ \widetilde{g} \ $
and with the {\itshape stable} framing determined by
$\ \bar{g}^{\#} \, $,
we obtain the framed bordism class
\renewcommand{\theequation}{\arabic{section}.\arabic{equation}'}
\addtocounter{equation}{-2}
\begin{equation}\label{equ:3.1Strich}
  \omega_{\infty}(*, f) \; = \; [ C, \, \widetilde{g},\, \bar{g}_{\infty} ] \; \in \; \Omega^{fr}_{m-n}(\Omega Y) \; = \; \lim_{r \rightarrow \infty} \pi_{m+1}(S^{n+r} \wedge (\Omega Y)^+ )
\end{equation}
\numberwithin{equation}{section}
\addtocounter{equation}{1}
which was discussed in detail in \cite{K2}.

Now let
$\ B \subset Y \ $
be a smoothly embedded compact $n$--ball with center point
$\ * \ $
such that
$\ y_0 \ $
lies in the boundary sphere
$\ \partial B \,$
and
$\ B \ $
contains the image of
$\ \gamma \,$.
We obtain a pinching map
\begin{equation*}
  \pinch \, \colon \ Y \; = \; (Y- \mathring{B} ) \cup_{\partial B} B \ \longrightarrow \ Y \diagup \partial B \cong S^n \,\vee \, Y
\end{equation*}
which collapses
$\ \partial B \ $
to a point.

\begin{thm}\label{thm:3.3}
  If
  $\ Y\ $
  is a simply connected, oriented $n$--dimensional manifold,
  $\ n \geq 2 ,\ $
  the diagram of homomorphisms
  \begin{equation*}
    \xymatrix @R=1.0cm @C=2.1cm {
    \pi_m(Y)	\ar[d]_-{\omega^{\#}(*,-)}	\ar[rd]^-{\qquad\qquad\quad \pinch_* - \incl_{Y*} \circ\; p_{2_*} \circ\; \pinch_*}	&
    \\
    \pi_m( S^n \wedge ((\Omega Y)^+ ))	\ar[r]_-{j_* \comp \kappa}^-{\cong}	&
    \Ker_{m,n}(Y) \; \subset \; \pi_m (S^n \vee Y )
    }
  \end{equation*}
  (compare \ref{equ:2.7} and \ref{prop:2.8}) commutes.
  (Here
  $\ \kappa \ $
  denotes the involution induced by
  $\ (-1) \cdot$~
  identity map on
  $\ \R^n \; \subset \; S^n \; = \; \R^n \cup \lbrace \infty \rbrace \, $.)
\end{thm}

\begin{proof}
  The corresponding result for general
  $\ Y \ $
  (with arbitrary finite fundamental group) was established in \cite{K3}, theorem 7.2.
  Here we give a different geometric proof for the special case
  $\ Y \, = \, S^n \ $
  (which is relevant for spherical space forms).
  
  Given
  $\ [f] \; \in \; \pi_m(S^n) ,\ $
  we may assume that
  $\ f \ $
  has the following standard form: there exists a smoothly embedded tubular neighbourhood
  \begin{equation*}
    T \; = \; C \times B^n \ \subset \ \R^m \ \subset \ S^m
  \end{equation*}
  such that
  $\ f \vert T \ $
  is the projection
  \begin{equation*}
    C \times B^n \; \longrightarrow \; B^n \diagup \partial B^n \; = \; S^n \; = \; \R^n \cup \lbrace \infty \rbrace
  \end{equation*}
  and
  $\ f(x) \ = \ y_0 \ = \ \infty\ $
  for all
  $\ x \; \notin \; T $.
  In the spirit of Pontryagin--Thom we may interpret
  $\ \pinch_*([f]) \ $
  by the framed link
  $\ C \; \nonComm \; C' \, \subset \R^m \ $
  consisting of the (neighbouring "parallel") components
  $\ C \; = \; C \times \lbrace 0 \rbrace \; = \; f^{-1}( \lbrace * \rbrace ) \ $
  and
  $\ C' \; =\; C \times \lbrace z_0 \rbrace \ $
  for some
  $\ z_0 \; \in \;\mathring{B}^n \setminus \lbrace 0 \rbrace $.
  Then
  $\ \incl_{Y*} \; \comp \; p_2 \; \comp \; \pinch_*([f]) \ $
  corresponds to the translated framed submanifold
  $\ C'' \; = \; C' + v_0 \; \subset \; \R^m ,\ $
  pushed away by some big vector
  $\ v_0 \; \in \; \R^m $,
  so that it does not link with
  $\ C \times B^n \ $
  anymore.
  
  Consider the homotopy
  \begin{equation}\label{equ:3.4}
    G' \; \colon \; C' \times I \; \longrightarrow \; \R^m , \quad G'(x, t) \; := \; x + t v_0 ,
  \end{equation}
  and the embedding
  \begin{equation*}
    E  \; \colon \; C' \times I \; \subset \R^m \times I , \qquad E(x, t) \; := \; ( x + t v_0, t ) ,
  \end{equation*}
  $\ ( x, t ) \; \in \; C' \times I $.
  We may assume that
  $\ * \; \in \; S^n \ $
  is a regular value of
  $\ f \comp G' \ $
  so that
  $\ E( C' \times I )\ $
  intersects
  $\ C \times I \ $
  transversely in an embedded submanifold
  \begin{equation*}
    K \; \subset \; C \times ( 0, 1 ) \; \subset \; \R^m \times ( 0, 1 ) \, .
  \end{equation*}
  Pick
  $\ \delta > 0 \ $
  such that the $\delta$--neighbourhood of
  $\ K \ $
  is still embedded, and so is the
  $\ \delta \cdot t$--neighbourhood of
  $\ (c, t) \; \in \; K \, ,$
  growing as
  $\ t > 0 \ $
  increases from one intersection of
  $\ E \ $
  with
  $\ \lbrace c \rbrace \times I \ $
  to a higher one, for any
  $\ c \; \in \; C \, $.
  Now remove the
  $\ \delta \cdot t$--ball
  $\ B' \; \subset \; E(C' \times I ) \ $
  around each point
  $\ (c, t ) \; \in \; K \ $
  and replace it by the cylinder
  $\ \partial B' \times [ t, 1 ] \; \subset \; \R^m \times I \, $.
  \addtocounter{figure}{4}
  \begin{figure}[htb]
    \centering
    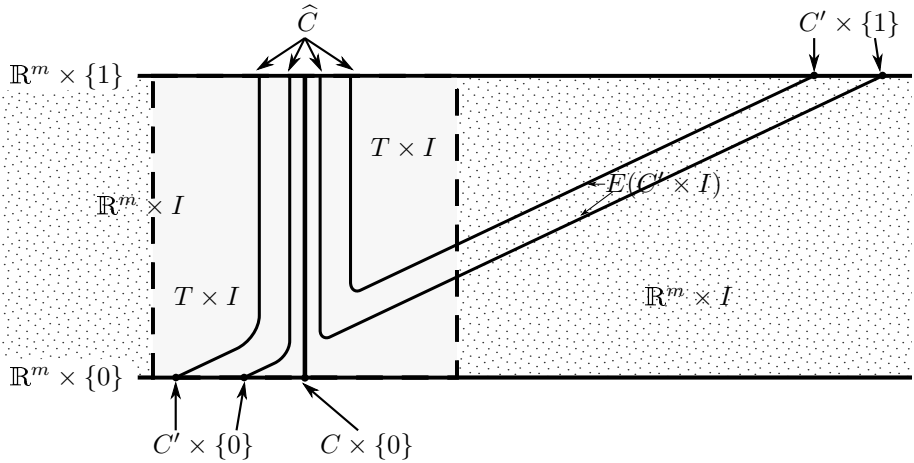
    \caption{The bordism which proves theorem \ref{thm:3.3} in case $\, Y \ =\ S^n$.}
    \label{fig:figure3.3alt}
  \end{figure}\stepcounter{equation}
  After smoothing corners we obtain an embedded framed bordism in
  $\ \left( \R^m \setminus C \right) \times I \ $
  from
  $\ C' \times \lbrace 0 \rbrace \ $
  to the disjoint union of
  $\ C' \times \lbrace 1 \rbrace \ $
  with a framed submanifold
  $\ \widehat{C} \; \subset \; \R^m \times \lbrace 1 \rbrace \ $
  which lies in the $\delta$--neighbourhood of
  $\ C \times \lbrace 1 \rbrace \ $
  (see figure \ref{fig:figure3.3alt}).
  But the link
    $\ C \; \nonComm \; \widehat{C} \, \subset \R^m \ $
  represents
  $\ j_* \comp \kappa (\omega^{\#}(*, f)) \,$.
  This follows from proposition \ref{prop:2.12} if we construct
  $\ \omega^{\#}(*, f) \ $
  using the homotopy
  $\ G' \ $
  (cf. (\ref{equ:3.4})) as well as the straight path between
  $\ 0 \ $
  and
  $\ z_0 \ $
  in
  $\ B^n \ $
  and a local isotopy along this path.
  Therefore
  $\ \pinch_*([f]) \ = \ \incl_{Y*} \, \comp \, p_{2*} \, \comp \, \pinch_*([f]) \, + \, j_* \comp \kappa (\omega^{\#}(*, f)) $.
\end{proof}

For all
$\ m, q \geq 1 \ $
there exists a canonical decomposition
\begin{equation}\label{equ:3.6}
  \Ker_{m, q}(Y) \; = \; \pi_m(S^q) \, \oplus \, \pi_m( S^q \; \flat \; Y )
\end{equation}
where
$\ S^q \; \flat \; Y \ $
denotes the homotopy fiber of the inclusion
$\ S^q \vee Y \; \subset \; S^q \times Y \ $
(cf. [G](9) and \cite{CLOT}, 6.7).

\providecommand{\KoschScript}[1]{\ensuremath{\mathit{#1}}}
\begin{cor}\label{cor:3.7}
  Let
  $\ B \ $
  be a compact $n$--ball embedded in the $1$--connected oriented manifold
  $\ Y \ $
  (as in \ref{thm:3.3}).
  For all
  $\ [f] \; \in \; \pi_m(Y) \ $
  \begin{equation*}
    j_* \comp \kappa(\omega^{\#}(*, f)) \ 
    = \ \left( \left[ \coll \comp f \right], \, H_{\KoschScript{C}}(f) \right) \; 
    \in \; \pi_m(S^n) \, \oplus \, \pi_m(S^n \; \flat \; Y )
  \end{equation*}
  where
  $\ b \ $
  and
  $\ H_{\KoschScript{C}}\ $
  are defined as a (\ref{equ:1.8}) and (\ref{equ:1.9}).
\end{cor}

Thus our basic coincidence invariant
$\ \omega^{\#}(*, f) \ $
turns out to be an enriched Hopf--Ganea invariant.
For a proof and further details see \cite{K3}, (63)--(65).

\begin{cor}\label{cor:3.8}
  For all
  $\ r \, = \, 0, 1, \ldots, \infty \ $
  and
  $\ [f] \; \in \; \pi_m(Y) \ $
  we have
  \begin{equation*}
    j_*(E^r(\kappa(\omega^{\#}(*, f) ))) \ 
    = \ \left( \, E^r( \left[ \coll \comp f \right]\right), \, e^r( H_{\KoschScript{C}}(f) )\, ) \ 
    \in \ \pi_{m+r}( S^{n+r} ) \, \oplus \, \pi_{m+r}( S^{n+r} \; \flat \; Y ) \, .
  \end{equation*}
  
  In particular,
  $\ \omega_r(*,f) \; = \; 0 \ $
  if and only if
  $\ E^r([b \comp f]) \; = \; 0 \ $
  and
  $\ e^r(H_{\KoschScript{C}}(f)) \; = \; 0 $.
\end{cor}

\begin{proof}
  This follows from the fact that
  $\ j_* \ $
  is injective and compatible with suspensions (cf. theorem \ref{thm:2.16}\ref{thm:2.16a}).
  Moreover,
  $\ \omega_r \; = \; E^r \comp \omega^{\#} \ $
  (cf. \ref{equ:3.1}) agrees with
  $\ E^r \comp \kappa \comp \omega^{\#} \ $
  up to an involution on
  $\ \pi_*(S^{n+r} \wedge (\Omega Y)^+ ) \ $
  of the form
  $\ (d \wedge \id)_* \ $
  where the map
  $\ d \; \colon \; S^{n+r} \; \longrightarrow \; S^{n+r} \ $
  has degree
  $\ (-1)^n$.
\end{proof}

\section{Computing Nielsen numbers}\label{sec:4}

\indent

In this section we prove theorem \ref{thm:1.11}
and proposition \ref{prop:1.12}

Let
$\ p \; \colon \; \widetilde{Y} \; \longrightarrow \; Y \ $
be a universal covering map of the $n$--dimensional manifold
$\ Y \, , n \geq 2 \, $,
and pick a basepoint
$\ \widetilde{y}_0 \; \in \; \widetilde{Y} \ $
such that
$\ p(\widetilde{y}_0) \; = \; y_0 \, $.
Also denote the number of pathcomponents of the loop space
$\ \Omega Y \; = \; \Omega(Y, y_0) \ $
by
\begin{equation*}
  k \; := \; \# \pi_1(Y) \; = \; \# \pi_0(\Omega Y) \, ,\ \ 1 \leq k \leq \infty \, .
\end{equation*}
Given homotopy classes
$\ [f], [f_1], [f_2], \ldots \; \in \; \pi_m(Y, y_0) \, , \ m \geq 2 \, $,
let
$\ [\widetilde{f}], [\widetilde{f}_1], [\widetilde{f}_2], \ldots \; \in \; \pi_m(\widetilde{Y}, \widetilde{y}_0) \ $
be their liftings.

The Nielsen number
$\ N_r(f_1, f_2) \ $
is extracted from the coincidence data
$\ ( i_r, \widetilde{g}, \bar{g}_r ) \ $
of a generic coincidence manifold
$\ C \ $
(cf. \ref{equ:1.1}) as follows,
$\ r = 0, 1, \ldots, \infty \, $.
Since the domain of
$\ f_1, f_2 \ $
is a sphere,
$\ \widetilde{g} \ $
maps
$\ C \ $
into the loop space
$\ \Omega Y \ $
(after suitable homotopies; cf. \cite{K2}, 2.4,
and compare also section \ref{sec:3} above).
Thus
$\ C \ $
is the disjoint union of
the {\itshape Nielsen classes}
$\ C_A \; = \; \widetilde{g}^{-1}(A) \, , \ A \; \in \; \pi_0(\Omega Y) \, $.
Such a Nielsen class
$\ C_A \ $
is called nonessential or essential, according to whether or not the coincidence data
$\ ( i_r, \widetilde{g}, \bar{g}_r ) \, $,
when restricted to
$\ C_A \, $,
form a nullbordant triple.
{\bfseries By definition
$\boldsymbol{\ N_r(f_1, f_2) \; \in \; \lbrace 0, 1, \ldots, k \rbrace \ }$
is the number of essential Nielsen classes.}

Clearly Nielsen numbers do not depend on the choice of the local orientation of
$\ Y \ $
at
$\ y_0 \ $
and of the path
$\ \gamma \ $
which play a r\^ole in the construction of
$\ \omega_r(f_1, f_2) \ $
(compare the proof of corollary \ref{cor:3.7}).

\begin{proof}[Proof of theorem \ref{thm:1.11}]\label{proof:thm:1.11alt}
  In the case 1	
  of the theorem our claim follows from proposition 1.3 in \cite{K4}.
  
  If
  $\ k \ $
  is finite and
  $\ n \geq 2 \, $,
  consider first the coincidence data of a pair of the form
  $\ (*, f) \ $
  where
  $\ * \neq y_0 \ $
  and
  $\ [f] \; \in \; \pi_m(Y, y_0) \, $.
  After suitable isotopies of
  $\ \widetilde{Y} \ $
  and deformations of
  $\ \widetilde{f} \ $
  we may assume that
  \begin{enumerate}[label=\roman*.)]
    \item there is a smoothly embedded open $n$--ball
      $\ \mathring{B} \; \subset \; \widetilde{Y} \setminus \lbrace \widetilde{y}_0 \rbrace \ $
      which contains all the points
      $\ p^{-1}( \lbrace * \rbrace ) \; = \; \lbrace \widetilde{*} \; = \; \widetilde{*}_1, \widetilde{*}_2, \ldots, \widetilde{*}_k \rbrace \; \subset \widetilde{Y} \ $
      over
      $\ * \; \in \; Y \,$;
    \item 
      $ \widetilde{f} \ $
      is smooth with regular value
      $\ \widetilde{*} \; = \; \widetilde{*}_1 \ $
      and maps a tubular neighbourhood
      \begin{equation*}
        \widetilde{C} \times \mathring{B} \; = \; \mathring{T} \; \subset \; \R^m \; \subset \; \R^m \cup \lbrace \infty \rbrace \; = \; S^m
      \end{equation*}
      of
      $\ \widetilde{C} \; := \; \widetilde{f}^{-1}(\lbrace \widetilde{*} \rbrace) \ $
      to
      $\ \mathring{B} \ $
      via the obvious projection; and
    \item
      $\ \widetilde{f}(S^m \setminus \mathring{T}) \; \subset \; \widetilde{Y} \setminus \mathring{B} \, $.
  \end{enumerate}
  Then the generic coincidence manifold
  $\ C \; := \; f^{-1}(\lbrace * \rbrace) \; = \; \widetilde{f}^{-1}( \lbrace \widetilde{*}_1, \widetilde{*}_2, \ldots, \widetilde{*}_k \rbrace ) \ $
  consists of the ("parallel") Nielsen classes
  \begin{equation*}
    \widetilde{C}_i \; = \; \widetilde{C} \times \lbrace \widetilde{*}_i \rbrace \; \subset \; \widetilde{C} \times \mathring{B} \; \subset \; S^m \, , \ i = 1, \ldots, k \, ,
  \end{equation*}
  which are simultaneously either all nonessential or essential, according as the coincidence data of
  $\ \widetilde{C} \; = \; \widetilde{C} \times \lbrace *_1 \rbrace \ $
  are nullbordant (or, equivalently
  $\ \omega_r(\widetilde{*}, \widetilde{f}) \; = \; 0 \, $)
  or not.
  Indeed, given a homotopy
  $\ G \; \colon \; \widetilde{C} \times I \; \longrightarrow \; S^m \ $
  from the inclusion
  $\ \widetilde{C} \; \subset \; S^m \ $
  to a constant map at
  $\ \infty \, $,
  base the construction of the $\omega$--invariant on the concatenation of
  $\ G \ $
  with the straight path
  $\ \widetilde{c}_i \ $
  from
  $\ \lbrace \widetilde{*}_i \rbrace \ $
  to
  $\ \lbrace \widetilde{*} \rbrace \ $
  in
  $\ \mathring{B} \, , i = 1, \ldots, k \ $
  (see the beginning of our section \ref{sec:3} above).
  Since the loops
  $\ p \comp \widetilde{c}_i \ $
  in
  $\ Y \ $ are pairwise nonhomotopic we get the Nielsen decomposition
  $\ C \; = \; \nonComm \widetilde{C}_i \ $
  with equally strong components.
  
  In contrast the coincidence data
  $\ [ i_r, \widetilde{g}, \bar{g}_r ] \ $
  of a pair of the form
  $\ (f, f) \ $
  have the special property that
  $\ \widetilde{g} \ $
  is homotopic to a constant map (cf. \cite{K3}, (21)).
  Thus the $\omega$--invariants of
  $\ (f, f) \ $
  and
  $\ (\widetilde{f}, \widetilde{f} ) \ $
  are equally strong and nontrivial precisely if the pathcomponent of the trivial loop in
  $\ \Omega Y \ $ corresponds to an essential Nielsen class.
  
  Next consider an arbitrary pair
  $\ (f_1, f_2) \, , \ [f_1], [f_2] \; \in \; \pi_m(Y, y_0) \, $.
  Use the chosen path
  $\ \gamma \ $
  from
  $\ y_0 \ $
  to
  $\ * \ $
  (cf. § \ref{sec:3}) and a small neighbourhood of the basepoint
  $\ \infty \ $
  in
  $\ S^m \ $
  to deform
  $\ f_1 \ $
  to a map
  $\ f_1' \; \colon \; (S^m, \infty ) \; \longrightarrow \; (Y, *) \, $.
  According to (\ref{equ:3.2})
  \begin{equation*}
    \left( \omega_r(f_1, f_2) = \right) \ \omega_r(f_1', f_2) \; = \;
    \omega_r(f_1', f_1) - \omega_r(*, f_1) + \omega_r(*, f_2) \, .
  \end{equation*}
  Applying our previous discussion to
  $\ [f] \; = \; [f_1] - [f_2] \,$,
  we see that the nontrivial elements of
  $\ \pi_0(\Omega Y) \; \cong \; \pi_1(Y) \ $
  yield essential Nielsen classes if and only if
  \begin{equation}\label{equ:4.1}
    \omega_r(\widetilde{*}, \widetilde{f}_1) \; \neq \; \omega_r(\widetilde{*}, \widetilde{f}_2) \, .
  \end{equation}
  The trivial element of
  $\ \pi_0(\Omega Y) \ $
  contributes an essential Nielsen class precisely if
  \begin{equation}\label{equ:4.2}
    \omega_r(\widetilde{f}_1', \widetilde{f}_1) - 
    \omega_r(\widetilde{*}, \widetilde{f}_1) + 
    \omega_r(\widetilde{*}, \widetilde{f}_2) \; \neq \; 0 \, .
  \end{equation}
  
  Now assume that
  $\ \widetilde{Y} \ $
  allows a {\itshape fixed point free} selfmap
  $\ a \, $.
  It is freely homotopic to a basepoint preserving map
  \begin{equation*}
    a^{\hochdot} \; \colon \; ( \widetilde{Y}, \widetilde{y}_0 ) \; \longrightarrow \; ( \widetilde{Y}, \widetilde{y}_0 ) \, .
  \end{equation*}
  Then
  $\ \omega_r(\widetilde{f}_1', a^{\hochdot} \comp \widetilde{f}_1) \; = \; \omega_r(\widetilde{f}_1, a \comp \widetilde{f}_1) \; = \; 0 \ $
  and
  \begin{equation}\label{equ:4.3}
    \omega_r(\widetilde{f}_1', \widetilde{f}_1) \; = \;
    \omega_r(\widetilde{f}_1', \widetilde{f}_1) - \omega_r(\widetilde{f}_1', a^{\hochdot} \comp \widetilde{f}_1)
    \; = \; \omega_r(\widetilde{*}, \widetilde{f}_1) - \omega_r(\widetilde{*}, a^{\hochdot} \comp \widetilde{f}_1) \, .
  \end{equation}
  \addtocounter{equation}{-2}
  \begin{subequations}
    \renewcommand{\theequation}{\theparentequation.\alph{equation}}
    Therefore condition (\ref{equ:4.2}) takes the form
    \begin{equation}\label{equ:4.2.a}
      \omega_r(\widetilde{*}, a^{\hochdot} \comp \widetilde{f}_1) \; \neq \; \omega_r(\widetilde{*}, \widetilde{f}_2) \, .
    \end{equation}
  \end{subequations}
  \addtocounter{equation}{1}
  This agrees with condition (\ref{equ:4.1}) if the Euler characteristic of
  $\ \widetilde{Y} \ $
  vanishes (e.g. when
  $\ n \ $
  is odd); indeed, a vector field without zeros yields a fixed point free selfmap
  $\ a \ $
  of
  $\ \widetilde{Y}\ $
  which is isotopic to the identity map
  $\ a^{\hochdot} \; = \; \id $.
  
  Finally apply the isomorphism
  $\ j_* \ $
  to the conditions (\ref{equ:4.1}) and (\ref{equ:4.2.a}) and use corollaries \ref{cor:3.7} and \ref{cor:3.8}.
  Also note that
  $\ H_{\KoschScript{C}}(\widetilde{f}_1) \; = \; H_{\KoschScript{C}}( a^{\hochdot} \comp \widetilde{f}_1 ) \ $
  since
  $\ \omega_r(\widetilde{f}_1', \widetilde{f}_1) \, $,
  and
  $\ j_*( \omega_r(\widetilde{f}_1', \widetilde{f}_1) ) \, $,
  resp., lie already in the subgroup
  $\ \pi_m(S^n) \ $
  of
  $\ \pi_m(S^n \wedge (\Omega \widetilde{Y})^+ ) \, $,
  and of
  $\ \Ker_{m,n}(\widetilde{Y}) \, $,
  resp. (compare (\ref{equ:3.6}) and \cite{K5}, 5.6).
  This completes the proof of theorem \ref{thm:1.11}.
\end{proof}

\begin{proof}[Proof of Proposition \ref{prop:1.12}]\label{proof:prop:1.12}
  We need to study only the arguments in the previous proof which deal with the case 2 of theorem \ref{thm:1.11}.
  All Nielsen classes are simultaneously essential or nonessential (i.e.
  $\ N_r(f_1, f_2) \; \in \; \lbrace 0, k \rbrace $)
  except possibly when
  $\ \omega_r(\widetilde{f}_1, \widetilde{f}_1) \, \neq \, 0 \ $
  (cf. (\ref{equ:4.2})).
  But in this case also
  $\ \omega_r(f_1, f_1) \, = \, E^r(\omega^{\#}(f_1, f_1) ) \ $
  and hence
  $\ \omega^{\#}(f_1, f_1) \ $
  are nontrivial.
  Thus all the restrictions listed in proposition \ref{prop:1.12} follow from \cite{K4}, proposition 1.3, and \cite{K6}, theorem 1.32.
\end{proof}

\section{Spherical space forms}\label{sec:5}

\indent

In this section we prove theorem \ref{thm:1.15} and its corollaries.

Let
$\ Y = S^n \diagup G \ $
be a spherical space form as in \ref{thm:1.15}; thus
$\ \widetilde{Y} \; = \; S^n $.
In view of the criteria (\ref{equ:4.1}) and (\ref{equ:4.2.a}) we need to apply theorem \ref{thm:3.3} only to (lifted) homotopy classes
$\ \widetilde{\varphi} \; \in \; \pi_m(S^n) $.
For the calculation of Nielsen numbers we may assume that the $n$--ball
$\ B \subset S^n \ $
(used in (\ref{equ:1.8}), (\ref{equ:1.9}) and in the construction of the pinching map in theorem \ref{thm:3.3}) is a suitable halfsphere, endowed with the standard orientation of
$\ S^n $.
Then
$\ b \sim \id \ $
in (\ref{equ:1.8}) and
\begin{equation}\label{equ:5.1}
  \pinch_*(\widetilde{\varphi}) \; = \; ( \iota_1 + \iota_2 ) \comp \widetilde{\varphi} \; \in \; \pi_m(S^n \vee S^n) \, ,
\end{equation}
where
$\ \iota_1 \ $
and
$\ \iota_2 \ $
are represented by the two obvious inclusions of
$\ S^n \ $
onto
$\ S^n \vee S^n $.
Using Hilton's choice of basic Whitehead products and applying his theorem A (in \cite{H})
we conclude that
\begin{equation}\label{equ:5.2}
  \left( \pinch_* - \; \iota_{2*} \comp p_{2*} \comp \pinch_* \right) (\widetilde{\varphi}) \; = \; 
  \iota_1 \comp \widetilde{\varphi}
  + \sum_{j \geq 1} w_j' \comp h_j'(\widetilde{\varphi})
  + \sum w_k'' \comp h_k''(\widetilde{\varphi}) \,;
\end{equation}
here the last two sums to the right involve those basic Whitehead products of
$\ \iota_1 \ $
and
$\ \iota_2 \ $
which contain
$\ \iota_1 \ $
precisely once (cf. \ref{equ:1.14}), and at least twice, resp.

Now according to theorems \ref{thm:2.16}\ref{thm:2.16a} and \ref{thm:3.3},
$\ \omega_r(\widetilde{*}, \widetilde{\varphi}) \; = \; E^r(\omega^{\#}(\widetilde{*}, \widetilde{\varphi}) ) \; = \; 0 \ $
or, equivalently,
$\ E^r(\kappa (\omega^{\#}(\widetilde{*}, \widetilde{\varphi}) ) ) \; = \; 0 \ $
(cf. \ref{thm:3.3}) if and only if the iterated partial suspension homomorphism
$\ e^r \ $
annihilates the right hand term in equation (\ref{equ:5.2}).
Denote this term by
$\ \tau $.
It vanishes precisely if its first summand and hence
$\ \widetilde{\varphi} \ $
itself does (by theorem A in \cite{H}).
When
$\ r \geq 1 \ $
our theorem \ref{thm:2.16}\ref{thm:2.16b} and corollary \ref{cor:2.17}, together again with Hilton's result (applied to
$\ S^{n+r} \vee S^n \ $)
imply that
\begin{equation*}
  e^r(\tau) \; = \; \iota_1 \comp E^r(\widetilde{\varphi})
  + \sum_{j \geq 1} e^r (w_j') \comp E^r(h_j' (\widetilde{\varphi}) ) \; = \; 0
\end{equation*}
if and only if
$\ E^r(\widetilde{\varphi}) \; = \; 0 \ $
and
$\ \pm E^r \comp h_j'(\widetilde{\varphi}) \; = \; 0 \ $
for all
$\ j \geq 1 $,
i.e.
$\ E^r \comp h(\widetilde{\varphi}) \; = \; 0 $.

Finally put
$\ \widetilde{\varphi} \; := \; [\widetilde{f}_1] - ( \pm \iota ) \comp [\widetilde{f}_2] \ $
and apply our criteria (\ref{equ:4.1}) and (\ref{equ:4.2.a});
note also that the antipodal map
$\ a \ $
is freely homotopic to a representative
$\ a^{\hochdot} \ $
of the generator
$\ (-1)^{n+1} \cdot \iota \ $
of
$\ \pi_n(S^n, y_0) $.
Theorem \ref{thm:1.15} and its corollaries \ref{cor:1.16} and \ref{cor:1.17} follow.
For corollary \ref{cor:1.19} compare also theorem 4.18 in \cite{BS}.
\qed

\begin{rem}\label{rem:5.3}
  The discussion following formula \ref{equ:4.3} implies, in particular, that
  \begin{equation*}
    h'([f]) \; = \; h'(( (-1)^{n+1} \cdot \iota ) \comp [f] )
  \end{equation*}
  for all
  $\ [f] \; \in \; \pi_m(S^n), \, m, n \geq 2 $.
  This has been used to simplify the criterion in Example \ref{exa:1.18}.
\end{rem}

\begin{proof}[Proof of corollary \ref{cor:1.20}]\label{proof:prop:1.20}
  As in the proof of proposition \ref{prop:1.12} we need to show only that
  $\ \omega_{\infty}(\widetilde{f}_1, \widetilde{f}_1 ) \ $
  vanishes (compare (\ref{equ:4.2})) or, {\itshape in the language of \cite{K2}, theorem 1.14}, that
  \begin{equation*}
    \widetilde{\omega}_j(\widetilde{f}_1, \widetilde{f}_1) \; := \; h_j(\widetilde{\omega}(\widetilde{f}_1, \widetilde{f}_1)) \; = \; 0 , \ \ j= 1, 2, \ldots \,.
  \end{equation*}
  When
  $\ j \geq 2 \ $
  the 'Hopf invariant component'
  $\ h_j(\widetilde{\omega}(\widetilde{f}_1, \widetilde{f}_1)) \ $
  is indeed trivial in this selfcoincidence situation (since
  $\ \widetilde{\omega}(\widetilde{f}_1, \widetilde{f}_1) \ $
  lies already in the subgroup
  $\ \Omega^{fr}_{m-n} \ $
  of
  $\ \Omega^{fr}_{m-n}(\Omega Y) $.
  Similarly
  \begin{equation*}
    \widetilde{\omega}_1(\widetilde{f}_1, \widetilde{f}_1) \; = \; \widetilde{deg}_1(\widetilde{f}_1) \pm \widetilde{deg}_1(\widetilde{f}_1) \; \in \; \pi_{m-n}^S
  \end{equation*}
  (cf. \cite{K2}, theorem 1.14) vanishes in view of our assumption
  $\ 2 \cdot \pi_{m-n}^S \; = \; 0 $.
\end{proof}

Note that this assumption cannot be dropped.
E.g. if
$\ n = 4, 8, 12, 14, 16 \text{ or } 20 \ $
then there exist infinitely many homotopy classes
$\ [f] \; \in \; \pi_{2n-1}(\RP(n)) \ $
such that
$\ \omega_{\infty}(f,f) \; \neq \; 0 \ $
or, equivalentely,
$\ N_{\infty}(f,f) \; = \; 1 $.
Indeed, apply corollary \ref{cor:1.17} to desuspensions of elements of order
$ \ > 2 \ $
in
$\ \pi_{n-1}^S \ $
(see also \cite{K4}, example 1.26).

\section{Examples}\label{sec:6}

In this section we use theorem \ref{thm:1.15} to establish the claims in examples \ref{exa:1.21}, \ref{exa:1.22}, \ref{exa:1.25} and \ref{exa:1.27}, as well as Proposition \ref{prop:1.24} and Table \ref{tab:1.26}.

The first claim in example \ref{exa:1.21} follows from theorem \ref{thm:1.15} or its corollaries since
\begin{equation*}
  E^{\infty} \comp h \, = \, (E^{\infty}, E^{\infty} \comp h') \, \colon \, \pi_m(S^n) \, \longrightarrow \, \pi^S_{m-n} \oplus \pi^S_{m-2n+1}
\end{equation*}
(cf. \ref{equ:1.13} ff) is injective here.
Indeed, in the stable range
$\ m \, \leq \, 2n - 2 \ $
already
$\ E^{\infty} \ $
is an isomorphism;
when 
$\ m \, = \, 2n - 1 \ $
the needed injectivity follows from the exact EHP--sequence
\begin{equation*}
  \xymatrix @R=1.5cm @C=2.0cm {
    \Z \ar[r]^-{\cdot [\iota_n, \iota_n] }	&
    \pi_{2n - 1}(S^n) \ar[r]^-{ E^{\infty} }	&
    \pi^S_{n - 1}
  }
\end{equation*}
(cf. \cite{W}, Ch. XII, (2.3) and (2.4)) and from the fact that the classical Hopf invariant
\begin{equation*}
  H \, \colon \, \pi_{2n-1}(S^n) \xrightarrow{E^{\infty} \comp h'_1} \pi^S_0 \, \equiv \, \Z
\end{equation*}
(cf. \cite{W}, Ch. XI, (8.17)) takes the value
$\ 2 \ $
on
$\ [\iota_n, \iota_n] \ $
(cf. \cite{W}, Ch. XI, (2.5)).

According to the generalized "Wecken theorem" 1.10 in \cite{K2} the minimum number
$\ \MCC \ $
agrees always with
$\ N_0 \, \equiv N_{\infty} \ $
when
$\ m < 2n - 2 \, $.
The remaining claims in Example \ref{exa:1.21} follow from \cite{K6}, theorems 1.12, 1.27 and 1.29, and from \cite{KR}, theorems 1.13 and 1.16. \qed

Next let us prove the claims in Example \ref{exa:1.22}.
When
$\ n \, \geq \, 2 \ $
is even and
$\ m \, \leq n+3 \, $,
then
$\ n \, \leq \, 2n -1 \ $
(and hence
$\ E^{\infty} \comp h \ $
is injective by the preceding proof) or else
$\ n \, = \, 2 \ $
and
$\ m \, = \, 4 \ $
or
$\ 5\ $
(and then already
$\ E^{\infty} \ $
alone is injective
(cf. \cite{T}, Propositions 5.3, 5.6 and Theorem 14.1,i).
Thus again all Nielsen numbers agree among themselves, and also with
$\ \MCC \ $
(by \cite{K6}, theorems 1.12 and (the last claim in) 1.19). \qed

In view of theorem \ref{thm:1.15} suspensions of the Hopf--Hilton homomorphism
$\ h \ $
(cf. \ref{equ:1.13} ff) play possibly a decisive role also in arbitrary dimensions
$\ m, \, n \,$.

\begin{lem}\label{lem:6.1}
  $ \left( \#\Ker\left(E^r \comp h \right)\right)_{r=0,1,\ldots} \ $
  is a nondecreasing sequence of \emph{finite} integers
  $\ \leq \# \pi_m(S^n) \,$.
  In fact, we have more: if
  $\ \# \pi_m (S^n) \, = \, \infty \, $,
  then
  $\ E^{\infty} \comp h \ $
  is injective and
  $\ \# \Ker(E^{r} \comp h) \, = \, 1 \ $
  for all
  $\ r \, \geq \, 0 \,$.
\end{lem}

The second claim is obvious when
$\ m \, = \, n \ $
and was established in the preceding discussion when
$\ m \, = \, 2n -1 \ $
and
$\ n \, \equiv \, 0(2) \, $.

\begin{proof}[Proof of Proposition \ref{prop:1.24}]\label{proof:prop:1.24}
  If condition \ref{prop:1.24:2} is satisfied then
  $\ \Ker(E^r \comp h) \, = \, \Ker(E^{r+1} \comp h) \ $
  by Lemma \ref{lem:6.1} and
  $\ E \ $
  is injective, when restricted to
  $\ \operatorname{Im}(E^r \comp h) \, $.
  Hence the criteria for
  $\ N_r \ $
  and
  $\ N_{r+1} \ $
  in theorem \ref{thm:1.15} agree and condition \ref{prop:1.24:1} holds.
  In turn this implies \ref{prop:1.24:4} and \ref{prop:1.24:3}.
  
  On the other hand, if \ref{prop:1.24:2} is not satisfied, then
  $\ \pi_m(S^n) \ $
  is finite by lemma \ref{lem:6.1} and contains a class
  $\ [\widetilde{f}] \ $
  such that
  \begin{equation*}
    E^r \comp h([\widetilde{f}]) \; \neq \; 0 \; = \; E^{r+1} \comp h([\widetilde{f}]) \, .
  \end{equation*}
  Thus according to theorem \ref{thm:1.15}
  \begin{equation*}
    N_r\lt\lt p \comp \rt [\widetilde{f}], 0 \rt \; \neq \; 0 \; = \; N_{r+1}\lt\lt p \comp \rt [\widetilde{f}], 0 \rt
  \end{equation*}
  and therefore
  $\ \#^0_r(m,Y) \, < \#^0_{r+1}(m,Y) \,$.
  This completes the proof.
\end{proof}

\vspace*{3mm}
Next we compute the cardinalities
$\ \#_r^i(m, Y) \ $
(cf. \ref{equ:1.23}) in a particularly simple special case.
\vspace*{3mm}

\begin{lem}\label{lem:6.2}
  Given
  $\ m, n \, \geq \, 2, \; n \text{ even} \, $,
  assume that
  $\ h' \, \equiv\, 0 \ $
  on
  $\ \pi_m(S^n) \ $
  (cf. \ref{equ:1.13Strich}).
  Consider the iterated suspension homomorphism
  $\ E^r \, \colon \, \pi_m(S^n) \, \longrightarrow \, \pi_{m+r}(S^{n+r}) \,$,
  and the (finite) cardinality
  $\ Q_r \, := \, \#\lbrace \alpha \in E^r(\pi_m(S^n)) \; \vert \; 2 \alpha \, = \, 0 \rbrace, \ 0 \, \leq \, r \, \leq \, \infty \,$.
  If
  $\ Y \, = \, S^n \diagup \Z_2 \ $
  as in theorem \ref{thm:1.15}, then
  \begin{align*}
    \#_r^0(m, Y) \, &= \, Q_r \cdot (\# \Ker E^r)^2 \, ; \\
    \#_r^1(m, Y) \, &= \, 2 \cdot \lt\lt \# E^r(\pi_m(S^n) \rt - Q_r \rt \; \cdot \; \lt \# \Ker E^r \rt^2 ; \\
    \#_r^2(m, Y) \, &= \, \lt \# \lt \pi_m(S^n) \rt \rt^2 - \#_r^0(m, Y) - \#_r^1(m, Y) .
  \end{align*}
  (All these cardinalities are finite except when
  $\ m \, = \, n \, $;
  in this case
  $\ \#_r^0(m, Y) \, = \, 1 \, $,
  but
  $\ \#_r^1(m, Y) \, = \, \#_r^2(m, Y) \, = \, \infty \,$).
  
  In particular, the number of pairs
  $\ \lt [f_1], [f_2] \rt, \; f_i \, \colon \, S^m \, \longrightarrow \, S^n \diagup \Z_2 \, , \ i\, = \, 1, 2 \, $,
  such that
  $\ N_0(f_1, f_2) \, = \, 0 \, $,
  is equal to
  $\ Q_0 \, $,
  i.e. to the number of elements of
  $\ \leq \, 2 \ $
  in
  $\ \pi_m(S^n) \, $.
\end{lem}

\begin{proof}
  Our assumption
  $\ h' \, \equiv \, 0 \ $
  simplifies the criteria in theorem \ref{thm:1.15} considerably and implies also that
  $\ (- \iota ) \comp [ \widetilde{f}_2 ] \, = \, - [ \widetilde{f}_2 ] \ $
  (cf. \cite{W}, Ch. XI, (8.12)).
  Thus the set of pairs
  $\ \lt [f_1], [f_2] \rt \; \in \; \pi_m(Y)^2 \ $
  such that
  $\ N_r(f_1, f_2) \, = \, 0 \ $
  (or $\ = \, 1 \, $, resp.) is characterized by the following conditions:
  \begin{enumerate}[label=(\roman*)]
    \item\label{proof:item_1}
      $ E^r([ \widetilde{f}_2 ]) \ $
      is an element of order
      $\ \leq 2 \, $,
      or not, resp., in
      $\ E^r(\pi_m(S^n)) \, $;
      and
    \item\label{proof:item_2}
      $ [\widetilde{f}_1] \, \in \, (E^r)^{-1} \lbrace \pm E^r([f_2]) \rbrace \, $.
  \end{enumerate}
  The 'number' of elements
  $\ [f_2] \, \in \, \pi_m(Y) \ $
  satisfying condition \ref{proof:item_1} is
  $\ Q_r \cdot \# \Ker E^r\, $,
  and
  $\ \lt \lt \# \operatorname{Im} E^r \rt - Q_r \rt \, \cdot \, \# \Ker E^r \, $,
  resp..
  Each such
  $\ [f_2] \ $
  can be paired with as many as
  $\ \# \lbrace \pm E^r [\widetilde{f}_2] \rbrace \cdot \# \Ker E^r \ $
  homotopy classes
  $\ [f_1] \, \in \, \pi_m(Y) \ $
  in order to satisfy also condition \ref{proof:item_2}.
  
  If
  $\ m \, = \, n \, $,
  then
  $\ N_r (f_1, f_2) \, = \, 0 \, $,
  or $\ 1 \, $, or $\ 2 \, $,
  resp., according as
  $\ ([f_1], [f_2]) \, = \, 0 \, $,
  or lies in the remaining union of the diagonal and antidiagonal in
  $\ \pi_m(Y)^2 \, \cong \, \Z \times \Z \, $,
  or outside of this union, resp..
\end{proof}

Finally let us apply lemma \ref{lem:6.2} to {\bfseries the case
$\ \boldsymbol{m \, = \, 16, \ Y \, = \, S^6 \diagup \Z_2} \,$.
}
We use the computations in Toda's book \cite{T} (cf. theorems 7.3, 13.9 and the tables in Chapter XIV) as well as Serre's theorem (cf. e.g. \cite{T}, (13.1)).
\addtocounter{table}{2}
\begin{table}[H]
  \centering
\begin{xy}
  \xymatrix @R=0.5cm @C=0.0cm {
    &						&									&									&									&									&			&	\ar@{-}[dddddd]	&			&	\ar@{-}[dddddd]	&					&\\
    &\pi_{16+r}(S^{6+r})			&									&									&									&									&			&			&	\#\Ker(E^r)	&			&	E^r (\pi_{16}(S^6))		&\\
    \ar@{-}[rrrrrrrrrrr]&			&									&									&									&									&			&			&			&			&					&\\
    &\pi_{16}(S^6)	\ar@{->>}[d]_-{E}	&	\cong \ \Z_8	\ar@<2mm>[d]^-{\rotatebox{270}{$\cong$}}	&	\oplus \ \ \lbrace 0 \rbrace					&	\oplus \ \ \Z_2	\ar@<2mm>[d]^-{\rotatebox{270}{$\cong$}}	&	\oplus \ \ \Z_9	\ar@<2mm>@{->>}[d]				&			&			&	1		&			&	\Z_8 \oplus \Z_2 \oplus \Z_9	&\\
    &\pi_{17}(S^7)	\ar[d]_-{E}	&	\cong \ \Z_8	\ar@<2mm>[d]^-{\rotatebox{270}{$\cong$}}	&	\oplus \ \ \lbrace 0 \rbrace					&	\oplus \ \ \Z_2	\ar@<2mm>[d]^-{\rotatebox{270}{$\cong$}}	&	\oplus \ \ \Z_3	\ar@<2mm>[d]^-{\rotatebox{270}{$\cong$}}	&			&			&	3		&			&	\Z_8 \oplus \Z_2 \oplus \Z_3	&\\
    &\pi_{18}(S^8)	\ar@{->>}[d]_-{E}	&	\cong \ \Z_8	\ar[dr]_(.35){\cdot \pm 2}			&	\oplus \ \ \Z_8	\ar@<2mm>[d]^-{\rotatebox{270}{$\cong$}}	&	\oplus \ \ \Z_2	\ar@<2mm>[d]^-{\rotatebox{270}{$\cong$}}	&	\oplus \ \ \Z_3	\ar@<2mm>[d]^-{\rotatebox{270}{$\cong$}}	&	\oplus \ \ \Z_3	&			&	3		&			&	\Z_8 \oplus \Z_2 \oplus \Z_3	&\\
    &\pi_{19}(S^9)	\ar@{->>}[d]_-{E}	&	\cong \quad							&	\ \ \ \ \Z_8	\ar@<2mm>@{->>}[d]						&	\oplus \ \ \Z_2	\ar@<2mm>[d]^-{\rotatebox{270}{$\cong$}}	&	\oplus \ \ \Z_3	\ar@<2mm>[d]^-{\rotatebox{270}{$\cong$}}	&			&			&	6		&			&	\Z_4 \oplus \Z_2 \oplus \Z_3	&\\
    &\pi_{20}(S^{10})	\ar@{->>}[d]_-{E}	&	\cong \quad							&	\ \ \ \ \Z_4	\ar@<2mm>@{->>}[d]						&	\oplus \ \ \Z_2	\ar@<2mm>[d]^-{\rotatebox{270}{$\cong$}}	&	\oplus \ \ \Z_3	\ar@<2mm>[d]^-{\rotatebox{270}{$\cong$}}	&			&			&	12		&			&	\Z_2 \oplus \Z_2 \oplus \Z_3	&\\
    &\pi_{21}(S^{11})	\ar@{->>}[d]_-{E}	&	\cong \quad							&	\ \ \ \ \Z_2								&	\oplus \ \ \Z_2	\ar@<2mm>[d]^-{\rotatebox{270}{$\cong$}}	&	\oplus \ \ \Z_3	\ar@<2mm>[d]^-{\rotatebox{270}{$\cong$}}	&			&			&	24		&			&	\Z_2 \oplus \Z_3		&\\
    &\pi_{22+j}(S^{12+j}), \, j\geq 0		&	\quad\cong							&									&	\ \ \ \ \ \Z_2							&	\oplus \ \ \Z_3							&			&			&	24		&			&	\Z_2 \oplus \Z_3		&\\
    &						&									&									&									&									&			&	\ar@{-}[uuuuu]	&			&	\ar@{-}[uuuuu]	&					&
  }
\end{xy}
  \caption{The suspension homomorphisms on the groups
$\ \pi_{16+r}(S^{6+r}), \; r \, \geq\, 0 \, $,
as described by Toda \cite{T}.
The cyclic direct summands in the $\ i^{\mathrm{th}} \ $ row,
$\ i \, = \, 1, 2, 3, 4 \, $,
are generated by
$\ \nu_{6+r} \comp \sigma_{9+r} \, , \;
\sigma_{6+r} \comp \nu_{13+r} \, , \;
\eta_{6+r} \comp \mu_{7+r}
\ $
and
$\ \beta_1(6+r) \, $,
resp..}
  \label{tab:6.3}
\end{table}\addtocounter{equation}{1}

In particular, the suspension homomorphism from
$\ \pi_{15}(S^5) \ $
to
$\ \pi_{16}(S^6) \ $
is both onto and injective.
This implies not only that
$\ h' \, \equiv \, 0 \ $
on
$\ \pi_{16}(S^6) \, $,
but also that
$\ \MC(f_1, f_2) \, = \, \MCC(f_1, f_2) \, = \, N_0(f_1, f_2) \ $
for all pairs
$\ f_1, f_2 \, \colon \, S^{16} \, \longrightarrow \, S^6 \diagup \Z_2 \ $
(cf. theorems \ref{thm:1.28} and \ref{thm:1.30}; or else \cite{K3}, Corollary 6.10 and theorem 6.14 as well as \cite{K6}, theorem 1.19).
Moreover we can extract the explicit description of the groups
$\ \pi_{16 + r}(S^{6+r}), \  r=0, 1, \ldots \,$,
and of the relevant suspension homomorphism as listed in Table \ref{tab:6.3}.
E.g. it follows from \cite{T}, (4.4) and (7.19), that
\begin{equation*}
  E(\nu_8 \comp \sigma_{11}) \, = \, \pm 2 \, \sigma_9 \comp \nu_{16} \, .
\end{equation*}

Now the data in Table \ref{tab:1.26} and the claims in Example \ref{exa:1.25} follow immediatly from Lemma \ref{lem:6.2}, Table \ref{tab:6.3} and Proposition \ref{prop:1.24}.

Similarly, according to theorem \ref{thm:1.15}
\begin{equation*}
  \#^0_r(16, S^6) \, = \, \#\pi_{16}(S^6) \cdot \# \Ker E^r \, .
\end{equation*}
Therefore Table \ref{tab:6.3} (together with \cite{K3}, 6.10 and 6.14) yields also the claims in Example \ref{exa:1.27}. \qed

\section{The minimum numbers MC and MCC.}\label{sec:7}

\indent

In this section we discuss theorems \ref{thm:1.28}, \ref{thm:1.30} and some of their consequences.

Theorem \ref{thm:1.28} follows from \cite{K6}, Corollary 1.20, except when
$\ Y\ $
is a sphere.
But if
$\ Y \, = \, S^n \ $
and
$\ a^{\hochdot} \, \colon \, (S^n, y_0 ) \, \longrightarrow \, (S^n, y_0) \ $
is freely homotopic to the antipodal map and
$\ [f] \, := \, [f_1] - a^{\hochdot}_*[f_2] \,$,
then
\begin{equation*}
  \MCC(f_1, f_2) \, = \, \MCC(f, y_0) \, = \, N_0(f, y_0) \, = \, N_0(f_1, f_2) \,.
\end{equation*}
Indeed, if also
$\ n \, \geq \, 2 \ $
then
$\ ([f_1],[f_2]) \, = \, ([f], y_0) \, + \, (a^{\hochdot}_*[f_2], [f_2]) \ $
and
$\ \MCC(a \comp f_2, f_2) \, = \, 0 \,$;
moreover,
$\ \MCC(f, y_0) \, \leq \, 1 \ $
vanishes precisely when
$\ [f] \, = \, 0 \ $
or, equivalently, 
$\ N_0(f, y_0) \, = \, 0 \ $
(since by construction
$\ \omega^{\#}(f, y_0) \, \in \, \pi_m(S^n \wedge (\Omega S^n)^+) \ $
determines
$\ [f] \, \in \, \pi_m(S^n) \,$.)

If
$\ Y\, = \, S^1 \,$,
then our claim follows from theorem 1.13 in \cite{K2}.
This completes the proof of our theorem \ref{thm:1.28}. \qed

\begin{proof}[Proof of theorem \ref{thm:1.30}]\label{proof:thm:1.30}
  We may assume that
  $\ m \, \geq \, n \ $
  since otherwise
  $\ \pi_m(S^n) \, = \, 0 \ $
  and
  $\ \MC \, \equiv \, \MCC \, \equiv \, 0 \,$.
  The same holds if
  $\ m \, > \, n \, = \, 1 \,$.
  If
  $\ m \, = \, n \, = \, 1 \ $
  and we denote the mapping degree of
  $\ f_i \, \colon \, S^1 \, \longrightarrow \, S^1 \ $
  by
  $\ d^0(f_i), \, i \, = 1,2 \,$,
  and put
  $\ d^0 \, := \, d^0(f_1 \cdot f_2^{-1}) \, = \, d^0(f_1) - d^0(f_2) \,$, then
  \begin{equation*}
    \MC(f_1, f_2) \, = \, \MCC(f_1, f_2) \, = \, \left| d^0(f_1) - d^0(f_2) \right| \, < \, \infty
  \end{equation*}
  since
  $\ f_1 \cdot f_2^{-1} \ $
  is homotopic to the map
  $\ z \, \longrightarrow \, z^{d^0}\, , \; z \, \in \, S^1 \,$,
  whose roots of unity belong to pairwise different Nielsen classes.
  
  If
  $\ m \, > \, n \, = \, 2 \ $
  and
  $\ \MCC(f_1, f_2) \, < \, \infty \,$,
  then
  $\ \MC(f_1, f_2) \, = \, \MCC(f_1, f_2) \, = \, 0 \ $
  since each isolated coincidence point has an 'index' in
  $\ \pi_{m-1}(S^{n-1}) \, = \, 0 \ $
  and hence may be eliminated by small deformations (cf. \cite{K3}, (28)).
  If
  $\ m \, = \, n \, = \, 2 \,$,
  then claim \ref{thm:1.30}\ref{thm:1.30-item1} follows from \cite{J}, theorem 4.0.
  
  Now we can deduce the full claim \ref{thm:1.30-item2} in our theorem \ref{thm:1.30} from \cite{K3}, corollary 6.10, applied to
  $\ [f] \, := \, [f_1] - [f_2] \,$:
  just note that
  $\ \left| \MC(f_1, f_2) - \MC([f_1] - [f_2], y_0) \right| \, \leq \, \MC(f_2, f_2) \, \leq \, 1 \ $
  (cf. \cite{K3}, Proposition 6.2, and \cite{K6}, theorem 1.19).
  In order to complete also the proof of claim \ref{thm:1.30-item1} we may assume that
  $\ m, \; n \, \geq \, 3 \ $
  and -- in view of \cite{K3}, theorem 1.2 -- also that
  $\ \MC(f_1, f_2) \, \leq \, \# G \,$.
  If
  $\ \MC(f_1, f_2) \, \neq \, \MCC(f_1, f_2) \,$,
  then obviously
  $\ N_0(f_1, f_2) \, < \, \#G \ $
  and hence (by Corollary \ref{cor:1.16})
  $\ f_1 \, \sim \, f_2 \ $
  or
  $\ f_1 \, \sim \, a \comp f_2 \ $
  and therefore
  $\ \MC(f_1, f_2) \, = \, \MCC(f_1, f_2) \, = \, 0 \ $
  or
  $\ = \, 1 \ $
  (cf \cite{K6}, theorem 1.19).
  Contradiction.
\end{proof}

\begin{proof}[Proof of Corollary \ref{cor:1.32}]\label{proof:cor:1.32}
  (i)\label{proof:cor:1.32-item1} Assume that
  $\ n \ $
  is even and
  $\ m \, \leq \, 3n - 4 \,$.
  Then
  $\ \#G \, \leq \, 2 \ $
  and
  $\ h' \ $
  fits into the exact EHP--sequence (cf. \cite{W}, Ch. XII, 2.3)
  \begin{equation*}
    \pi_{m-1}(S^{n-1}) \, \overset{E}{\longrightarrow} \, \pi_m(S^n) \, \xrightarrow{H= h'} \, \pi_m(S^{2n-1}) \, \longrightarrow \ldots \ .
  \end{equation*}
  Indeed,
  $\ \pi_m(S^{2m-1}) \ $
  is stable and hence the Hopf--James invariant
  $\ H \ $
  agrees with the Hopf--Hilton invariant
  $\ h' \ $
  (cf. \cite{BS}, theorem 4.18).
  
  (ii)\label{proof:cor:1.32-item2} If
  $\ n \ $
  is odd then
  $\ \MCC(f_1, f_2) \, = \, N_0(f_1, f_2) \ $
  (cf. theorem \ref{thm:1.28}) is described in corollary \ref{cor:1.16}\ref{cor:1.16-item1}; here a
  $\,\sim \, \id \,$.
\end{proof}

Finally let us discuss Example \ref{exa:1.33}.
Here
$\ \underline{h}' \, \colon \, \pi_{m-1}(S^{n-1}) \, \longrightarrow \, \Z \ $
coincides with the classical Hopf invariant homomorphism (cf. \cite{W}, Ch. XI, 8.17) and is onto.
According to the first argument in section \ref{sec:6} above the Freudenthal suspension epimorphism
$\ E \ $
is injective when restricted to
$\ \Ker \underline{h}' \,$,
i.e. to the torsion subgroup of
$\ \pi_{m-1}(S^{n-1}) \,$.
An inspection of Toda's table I
(cf. \cite{T}, p.186) now shows us that
$\ E(\Ker(\underline{h}')) \ $
is a subgroup of index 2 of the cyclic group
$\ E(\pi_{m-1}(S^{n-1})) \, = \, \pi_m(S^n) \,$.
\qed

\vspace{1.5ex}
Since all closed surfaces but
$\ S^2\ $
and
$\ \RP(2)\ $
are aspherical the claims in Example \ref{exa:1.34} follow from
\begin{prop}\label{prop:7.1}
  Assume that
  $\ \pi_{m-1}(S^{n-1}) \, = \, 0 \ $
  where
  $\ m, n \, \geq \, 2 \,$.
  Then we have for all maps
  $\ f_1, f_2 \, \colon \, S^m \, \longrightarrow \, S^n \diagup G \,$:
  \begin{enumerate}[label=(\roman*)]
    \item\label{prop:7.1:1} If
      $\ f_1 \, \sim \, f_2 \,$,
      then
      $\ \MC(f_1, f_2) \, = \, \MCC(f_1, f_2) \, = \, N_r(f_1, f_2) \, = \, 0 \ $
      for all
      $\ r \, = \, 0, 1, \ldots \infty \,$.
    \item\label{prop:7.1:2} If
      $\ f_1 \, \not\sim \, f_2 \,$,
      then
      $\ \MC(f_1, f_2) \, = \, \infty \ $
      but
      $\ \MCC(f_1, f_2) \, = \, N_0(f_1, f_2) \, = \, \#G \,$.
  \end{enumerate}
\end{prop}

\begin{proof}
  If
  $\ \MC(f_1, f_2) \, < \, \infty \,$,
  the maps
  $\ f_i \ $
  may be deformed until they have only isolated coincidence points.
  Each of these can be removed by a further local deformation since its index (very similar to the index of a vector field at an isolated zero, cf. \cite{K3}, (28)) lies in
  $\ \pi_{m-1}(S^{n-1}) \, = \, 0 \,$.
  Thus
  $\ \MC(f_1, f_2) \ $
  vanishes (and so do
  $\ \MCC(f_1, f_2)\ $
  and the Nielsen numbers).
  According to theorem \ref{thm:1.30} this happens precisely when
  $\ [f_1'] \, = \, [f_2']\ $
  where
  $\ f_i' \ $
  is any basepoint preserving map freely homotopic to
  $\ f_i, \, i\, = \, 1,2 \,$.
  Hence
  $\ f_1 \, \sim \, f_2 \,$;
  in turn,
  $\ \MC(f_1, f_2) \, \leq \, 1 \ $
  whenever
  $\ f_1 \, \sim \, f_2 \ $
  (cf. \cite{K6}, theorem 1.19).
  The previous argument shows also that
  $\ \MC(f_2, f_2) \, = \, 0 \ $
  even when
  $\ f_1 \, \not\sim \, f_2 \,$.
  Thus the pairs
  $\ (f_1, f_2) \ $
  and
  $\ ([f_1'] - [f_2'], 0) \ $
  have the same minimum and Nielsen numbers (cf. \cite{K3}, 6.2).
  Claim \ref{prop:7.1:2} follows now from theorem \ref{thm:1.28} and corollary \ref{cor:1.16}.
\end{proof}

\vspace{1ex}
It is a curious consequence of the last proof that each map
$\ f \, \colon \, S^m \, \longrightarrow \, S^n \ $
is homotopic to its composite
$\ a \, \comp \, f \ $
with the antipodal map
$\ a \ $
whenever
$\ m, n \, \geq \, 2 \ $
and
$\ \pi_{m-1}(S^{n-1}) \, = \, 0 \,$.
Indeed, clearly
$\ \MC(f, \, a \comp f) \, = \, 0 \,$.

The papers [K1]--[K7] and [KR] are available online at\\
http://www.uni-siegen.de/fb6/rmi/topologie/publications.html\,.

\end{document}